\documentclass{amsart}

\usepackage{amssymb,amsfonts}
\usepackage[all]{xy}
\usepackage{enumerate}
\usepackage[noadjust]{cite}
\usepackage{hyperref}

\newtheorem{thm}{Theorem}[section]
\newtheorem{cor}[thm]{Corollary}
\newtheorem{prop}[thm]{Proposition}
\newtheorem{lem}[thm]{Lemma}

\newtheorem{quest}[thm]{Question}

\theoremstyle{definition}
\newtheorem{defn}[thm]{Definition}

\newtheorem{exmp}[thm]{Example}

\newtheorem{fact}[thm]{Fact}

\theoremstyle{remark}
\newtheorem{rem}[thm]{Remark}

\newcommand{\BTP}{\mathrm{BTP}}
\newcommand{\NBTP}{\mathrm{NBTP}}
\newcommand{\TP}{\mathrm{TP}}
\newcommand{\SOP}{\mathrm{SOP}}
\newcommand{\NTP}{\mathrm{NTP}}
\newcommand{\NSOP}{\mathrm{NSOP}}
\newcommand{\ATP}{\mathrm{ATP}}
\newcommand{\NATP}{\mathrm{NATP}}
\newcommand{\CTP}{\mathrm{CTP}}
\newcommand{\NCTP}{\mathrm{NCTP}}
\newcommand{\NIP}{\mathrm{NIP}}

\newcommand{\DLO}{\mathrm{DLO}}
\newcommand{\feq}{\mathrm{feq}}
\newcommand{\K}{\mathcal{K}}
\newcommand{\D}{\mathcal{D}}
\newcommand{\E}{\mathcal{E}}
\newcommand{\Power}{\mathcal{P}}
\DeclareMathOperator{\tp}{\mathrm{tp}}
\newcommand{\monster}{\mathbb{M}}
\newcommand{\pfc}{\mathrm{pfc}}
\DeclareMathOperator{\acl}{\mathrm{acl}}
\newcommand{\RCF}{\mathrm{RCF}}

\DeclareMathOperator{\Av}{\mathrm{Av}}

\def\Ind{\setbox0=\hbox{$x$}\kern\wd0\hbox to 0pt{\hss$\mid$\hss} \lower.9\ht0\hbox to 0pt{\hss$\smile$\hss}\kern\wd0} 

\def\Notind{\setbox0=\hbox{$x$}\kern\wd0\hbox to 0pt{\mathchardef \nn=12854\hss$\nn$\kern1.4\wd0\hss}\hbox to 0pt{\hss$\mid$\hss}\lower.9\ht0 \hbox to 0pt{\hss$\smile$\hss}\kern\wd0} 

\def\ind{\mathop{\mathpalette\Ind{}}} 

\def\nind{\mathop{\mathpalette\Notind{}}} 

\title{A New Kim's Lemma}
\thanks{N.R. was supported by NSF grant DMS-2246992 and an AMS-Simons Travel Grant.}

\date{\today}
\author{Alex Kruckman and Nicholas Ramsey}

\begin{document}

\begin{abstract} 
Kim's Lemma is a key ingredient in the theory of forking independence in simple theories. It asserts that if a formula divides, then it divides along every Morley sequence in type of the parameters. Variants of Kim's Lemma have formed the core of the theories of independence in two orthogonal generalizations of simplicity\textemdash namely, the classes of $\NTP_2$ and $\NSOP_1$ theories. We introduce a new variant of Kim's Lemma that simultaneously generalizes the $\NTP_2$ and $\NSOP_1$ variants. We explore examples and non-examples in which this lemma holds, discuss implications with syntactic properties of theories, and ask several questions. 
\end{abstract}

\maketitle

\setcounter{tocdepth}{1}
\tableofcontents

\section{Introduction} 

The simple theories are a class of first-order theories which admit a structure theory built upon a good notion of independence. Non-forking independence was introduced by Shelah~\cite{shelah1990classification} in the context of classification theory for stable theories, but was later shown to be meaningful in a broad class of unstable theories. Shelah's characterization of simple theories in terms of their saturation spectra~\cite{shelah1980simple}, together with Hrushovski's work on bounded PAC structures and structures of finite $S_{1}$-rank~\cite{hrushovski1991pseudo}, and the work of Cherlin and Hrushovski on quasi-finite theories~\cite{cherlin2003finite}, all made use of a circle of ideas concerning independence and amalgamation.   These ideas were subsequently distilled and consolidated into the core results of simplicity theory by Kim and Pillay~\cite{kim1998forking,kim1997simple}, organized around the good behavior of non-forking independence in this setting. A key ingredient in this theory is a result known as \emph{Kim's Lemma}, which establishes that, in a simple theory, a formula $\varphi(x;b)$ divides over a set $A$ if and only if $\varphi(x;b)$ divides along every Morley sequence over $A$ in $\tp(b/A)$.  Kim's Lemma says that dividing is always witnessed by ``generic'' indiscernible sequences and changes the existential quantifier in the definition of dividing (``\emph{there is} an $A$-indiscernible sequence such that...'') into a universal one (``\emph{for every} Morley sequence over $A$...'').  Kim later showed that Kim's Lemma characterizes the simple theories~\cite{kim2001simplicity}.

More recent developments have highlighted the centrality of Kim's Lemma to the theory of non-forking independence and its generalizations. In particular, the theories of independence in $\NTP_{2}$ and $\NSOP_{1}$ theories\footnote{As a consequence of Mutchnik's work in~\cite{mutchnik2022nsop2}, we now know that the properties $\NSOP_1$, $\NSOP_2$, and $\NTP_1$ are equivalent at the level of theories. In this paper, we primarily refer to $\NSOP_1$ theories (rather than to $\NSOP_2$ or $\NTP_1$ theories), since the notion of Kim-independence was originally developed in~\cite{kaplan2017kim} under this hypothesis.} are based on two orthogonal generalizations of Kim's Lemma. For $\NTP_{2}$ theories, the equivalence between dividing and dividing along all generic sequences is preserved, but this equivalence requires a stronger notion of genericity.  More specifically, Chernikov and Kaplan showed in~\cite{chernikov2012forking} that, in an $\NTP_{2}$ theory, a formula $\varphi(x;b)$ divides over a model $M$ if and only if $\varphi(x;b)$ divides along Morley sequences for every strictly $M$-invariant type extending $\tp(b/M)$. This variant of Kim's Lemma was shown to characterize $\NTP_2$ theories in~\cite{ChernikovNTP2}.  

On the other hand, in $\NSOP_{1}$ theories, the equivalence between dividing and dividing along generic sequences no longer holds in general.  Nonetheless, at the generic scale, there is an analogue of Kim's Lemma: a formula $\varphi(x;b)$ divides along \emph{some} generic sequence in $\tp(b/M)$ over a model $M$ if and only if it divides along \emph{every} such sequence.  More precisely, Kaplan and the second-named author introduced \emph{Kim-dividing}, which is defined so that a formula $\varphi(x;b)$ Kim-divides over a model $M$ if $\varphi(x;b)$ divides along some Morley sequence for a global $M$-invariant type extending $\tp(b/M)$. It was shown in~\cite{kaplan2017kim} that, in an $\NSOP_{1}$ theory, $\varphi(x;b)$ Kim-divides over $M$ if and only if it divides along Morley sequences for every global $M$-invariant type extending $\tp(b/M)$ and that, moreover, this variant of Kim's Lemma characterizes $\NSOP_{1}$ theories. 

We introduce a ``New Kim's Lemma'' that simultaneously generalizes the Kim's Lemmas for $\NTP_{2}$ and $\NSOP_{1}$ theories. The starting point is an observation about the Broom Lemma of Chernikov and Kaplan~\cite{chernikov2012forking}.  This lemma is the key step in showing that, in $\NTP_{2}$ theories, types over models always have global strict invariant extensions, which generate the generic sequences needed to get a Kim's Lemma for $\NTP_{2}$ theories. However, an inspection of the proof shows that this fact really bundles together two separate statements.  The first is that in $\NTP_{2}$ theories, Kim-dividing and forking independence coincide over models.  The second is that, in any theory whatsoever, types over models extend to global \emph{Kim-strict} invariant types, where Kim-strictness relaxes the non-forking independence condition required for strictness to one that only requires non-Kim-forking. See Theorem~\ref{thm:kimstrict} below. 

The statement of New Kim's Lemma, then, suggests itself (see Definition~\ref{def:nkl} below):  a formula $\varphi(x;b)$ Kim-divides over a model $M$ if and only if it divides along Morley sequences for every Kim-strictly $M$-invariant type extending $\tp(b/M)$. This variant of Kim's Lemma coincides with the Chernikov--Kaplan Kim's Lemma in $\NTP_{2}$ theories (since there, Kim-forking agrees with forking over models, and hence Kim-strict invariant types are strict invariant), and coincides with the Kaplan--Ramsey Kim's Lemma in $\NSOP_1$ theories (since there, Kim-independence is symmetric, so invariant types are automatically Kim-strict).

In Section~\ref{sec:diversity}, we survey the Kim's Lemmas of the past and introduce our New Kim's Lemma. We also observe that New Kim's Lemma implies that Kim-forking equals Kim-dividing at the level of formulas. In Section~\ref{sec:examples}, we show that our variant of Kim's lemma holds in some examples of interest, including parametrized dense linear orders and the two-sorted theory of an infinite dimensional vector space over a real closed field with a bilinear form which is alternating and non-degenerate or symmetric and positive-definite. Our choice of examples is motivated by the idea that structures obtained by ``generically putting together'' $\NTP_2$ and $\NSOP_1$ behavior should satisfy New Kim's Lemma. We show, however, that New Kim's Lemma does \emph{not} hold in the generic triangle-free graph, suggesting that it could serve as a meaningful dividing line among theories.

In Section~\ref{sec:syntax}, we try to relate New Kim's Lemma to syntactic properties of formulas. Our approach here reverses the usual order of explanation in neostability theory, which typically begins with a syntactic property (e.g., the tree property, $\SOP_{1}$, $\TP_{2}$) and then tries to establish a structure theory for theories without this property.  In contrast, we are starting with a structural feature and trying to find a way of characterizing it syntactically. We introduce a new combinatorial configuration, which we provisionally call the Bizarre Tree Property $(\BTP)$. The class of $\NBTP$ theories (those without $\BTP$) contains $\NTP_2$ and $\NSOP_1$, and all $\NBTP$ theories satisfy New Kim's Lemma. However, we do not obtain an exact characterization. 

The Antichain Tree Property $(\ATP)$, which was introduced in~\cite{ahn2021sop1} and developed in~\cite{ATP}, is another combinatorial configuration generalizing $\TP_2$ and $\SOP_1$. We observe that $\NBTP$ implies $\NATP$.  But it is not clear whether there is an implication in either direction between $\NATP$ and New Kim's Lemma, or whether $\NBTP$ and $\NATP$ are equivalent. Here is a diagram summarizing this state of affairs.
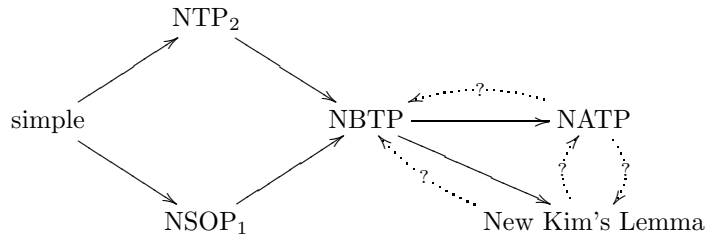
\[\xymatrix{
& \NTP_2\ar[dr]  & &  \\
\text{simple} \ar[ur] \ar[dr]  & &  \NBTP \ar[r] \ar[dr] & \NATP \ar@{.>}@/_1pc/|-{\text{?}}[l]\ar@{.>}@/^1pc/|-{\text{?}}[d] \\
& \NSOP_1\ar[ur]  & & \text{New Kim's Lemma} \ar@{.>}@/^1pc/|-{\text{?}}[ul]  \ar@{.>}@/^1pc/|-{\text{?}}[u]
}\]

While this paper was in preparation, two closely related preprints appeared. 
\begin{itemize}
\item In~\cite{kim2023remarks}, Kim and Lee establish a different variant of Kim's Lemma for $\NATP$ theories. Similarly to our work here, they do not prove that this Kim's Lemma characterizes $\NATP$. In the context of $\NATP$, they also study dividing along coheir sequences which are Kim-strict in the sense of this paper. 
\item In~\cite{hanson2023biinvariant}, Hanson studies a number of variants of Kim's Lemma which are related to ours. In particular, he succeeds in characterizing the class of $\NCTP$ theories by means of a variant of Kim's Lemma. Here $\CTP$ is the \emph{comb tree property} (which was introduced by Mutchnik in~\cite{mutchnik2022nsop2} under the name $\omega$-$\mathrm{DCTP}_2$). The class of $\NCTP$ theories contains the $\NBTP$ theories and its contained in the $\NATP$ theories.
\end{itemize}
At the moment, the $\NATP$ theories are the class beyond $\NSOP_{1}$ and $\NTP_{2}$ with the most developed syntactic theory; it would be very satisfying if these three approaches coincide. We conclude in Section~\ref{sec:questions} with several questions on where the theory might go from here. 

\section{Preliminaries}\label{sec:preliminaries}

Throughout, $T$ is a complete $L$-theory and $\monster\models T$ is a monster model. As usual, all tuples come from $\monster$, all sets are small subsets of $\monster$, and all models are small elementary submodels of $\monster$. 

When $\alpha$ is an ordinal, we view the set $\alpha^{<\omega}$ of all finite sequences from $\alpha$ as a tree, with the tree partial order denoted by $\trianglelefteq$. The root of the tree is the empty sequence $\langle\rangle$. For $\rho\in \alpha^\omega$ and $i< \omega$, $\rho|i\in \alpha^{<\omega}$ is the restriction of $\rho$ to $i$. We write $\eta^{\frown}\nu $ for concatenation of sequences. We write $\eta\perp\nu$ when $\eta$ and $\nu$ are incomparable in the tree order.  An \emph{antichain} is a set of pairwise incomparable elements. 

\subsection{Tree properties}

We will begin by recalling the definitions of a number of tree properties and the known implications between them.  The following three tree properties were introduced by Shelah under different names\footnote{$\TP$, $\TP_{1}$, and $\TP_{2}$ were first introduced under the rather cumbersome labels $\kappa_{\mathrm{cdt}}(T) = \infty$, $\kappa_{\mathrm{sct}}(T) = \infty$, and $\kappa_{\mathrm{inp}}(T) = \infty$, respectively.} in \cite{shelah1990classification} as part of his analysis of forking in stable theories.  He introduced the `tree property' terminology in \cite{shelah1980simple} and Kim subsequently dubbed the latter two as $\TP_{1}$ and $\TP_{2}$ in \cite{kim2001simplicity}.  

\begin{defn}
Let $\varphi(x;y)$ be a formula. 
\begin{enumerate}
\item We say $\varphi(x;y)$ has the \emph{tree property} ($\TP$) if there is $k < \omega$ and a tree of tuples $(a_{\eta})_{\eta \in \omega^{<\omega}}$ satisfying the following conditions:
\begin{enumerate}
\item For all $\rho \in \omega^{\omega}$, $\{\varphi(x;a_{\rho | i }) : i < \omega\}$ is consistent.
\item For all $\eta \in \omega^{<\omega}$, $\{\varphi(x;a_{\eta^{\frown}\langle j \rangle }) : j < \omega\}$ is $k$-inconsistent. 
\end{enumerate}
\item We say $\varphi(x;y)$ has the \emph{tree property of the first kind} ($\TP_{1})$ if there is a tree of tuples $(a_{\eta})_{\eta \in \omega^{<\omega}}$ satisfying the following conditions:
\begin{enumerate}
\item For all $\rho \in \omega^{\omega}$, $\{\varphi(x;a_{\rho | i }) : i < \omega\}$ is consistent.
\item For all $\eta, \nu \in \omega^{<\omega}$, if $\eta \perp \nu$, then $\{\varphi(x;a_{\eta}), \varphi(x;a_{\nu})\}$ is inconsistent. 
\end{enumerate}
\item We say $\varphi(x;y)$ has the \emph{tree property of the second kind} ($\TP_{2}$) if there is $k < \omega$ and an array $(a_{i,j})_{i, j < \omega}$ satisfying the following conditions:
\begin{enumerate}
\item For all $f : \omega \to \omega$, $\{\varphi(x;a_{i,f(i)}) : i < \omega\}$ is consistent. 
\item For all $i < \omega$, $\{\varphi(x;a_{i,j}) : j < \omega\}$ is $k$-inconsistent. 
\end{enumerate} 
\item We say $T$ is $\NTP$ ($\NTP_{1}$, $\NTP_{2}$) if no formula has $\TP$ ($\TP_{1}$, $\TP_{2}$, respectively) modulo $T$. An $\NTP$ theory is also called a \emph{simple theory}.  
\end{enumerate}
\end{defn}

The next property was introduced by D\v{z}amonja and Shelah in~\cite{dvzamonja2004maximality}.

\begin{defn}
\cite[Definition 2.2]{dvzamonja2004maximality} We say $\varphi(x;y)$
has the $1$-\emph{strong order property} ($\SOP_{1}$) if there
is a tree of tuples $(a_{\eta})_{\eta\in2^{<\omega}}$ satisfying the following conditions:
\begin{itemize}
\item For all $\rho \in2^{\omega}$, the set of formulas $\{\varphi\left(x;a_{\rho | i}\right): i<\omega\}$
is consistent. 
\item For all $\nu,\eta\in2^{<\omega}$, if $\nu^\frown\langle0\rangle\trianglelefteq\eta$
then $\left\{ \varphi\left(x;a_{\eta}\right),\varphi\left(x;a_{\nu^{\frown}\langle1\rangle}\right)\right\} $
is inconsistent. 
\end{itemize} 
$T$ is $\NSOP_{1}$ if no formula has $\SOP_{1}$ modulo $T$. 
\end{defn}

Our last property was introduced much more recently by Ahn and Kim in~\cite{ahn2021sop1}.

\begin{defn}\cite[Definition 4.1]{ahn2021sop1}
We say $\varphi(x;y)$ has the \emph{antichain tree property} ($\ATP$) if there is a tree of tuples $(a_{\eta})_{\eta \in 2^{<\omega}}$ satisfying the following conditions:
\begin{enumerate}
\item If $X \subseteq 2^{<\omega}$ is an antichain, then $\{\varphi(x;a_{\eta}) : \eta \in X\}$ is consistent.
\item If $\eta \trianglelefteq \nu \in 2^{<\omega}$, then $\{\varphi(x;a_{\eta}), \varphi(x;a_{\nu})\}$ is inconsistent. 
\end{enumerate}
$T$ is $\NATP$ if no formula has $\ATP$ modulo $T$. 
\end{defn}

\begin{fact} Here is a summary of the known implications, which are depicted in the diagram below. 
\begin{enumerate}
\item The simple theories are the intersection of the $\NTP_{1}$ and $\NTP_{2}$ theories, i.e., $T$ is simple if no formula has $\TP_{1}$ or $\TP_{2}$ modulo $T$. (\cite[Theorem III.7.11]{shelah1990classification}.)
\item A theory $T$ is $\NSOP_{1}$ if and only if it is $\NTP_{1}$. (\cite[Theorem 1.6]{mutchnik2022nsop2}\footnote{The theorem as stated in \cite{mutchnik2022nsop2} says that every $\NSOP_2$ theory is $\NSOP_1$. Prior to the appearance of this result, it was well-known that $\NSOP_1$ implies $\NSOP_2$ and $\NSOP_2$ is equivalent to $\NTP_1$. See, e.g., \cite{KIM2011698}.}.) 
\item The $\NATP$ theories (properly) contain both the $\NTP_{1}$ and $\NTP_{2}$ theories. (\cite[Propositions 4.4 and 4.6]{ahn2021sop1}.)
\end{enumerate}
\end{fact}
\[\xymatrix{
& \NTP_2 \ar[dr] & \\
\NTP = \text{simple} \ar[ur] \ar[dr]& & \NATP\\
& \NTP_1 = \NSOP_1\ar[ur] & 
}\]

\subsection{Forking and dividing} \label{sec:dividing}
In this section, we introduce a number of refinements of Shelah's notions of forking and dividing, based on the idea that, when a formula divides, it can be useful to study which indiscernible sequences it divides along. 

\begin{defn}
Suppose $\varphi(x;b)$ is a formula, $C$ is a set, and $I = (b_{i})_{i < \omega}$ is a $C$-indiscernible sequence in $\tp(b/C)$ (meaning that $b_i$ realizes $\tp(b/C)$ for all $i< \omega$). We say that $\varphi$ \emph{divides along} $I$ (over $C$) if $\{\varphi(x;b_{i}) : i < \omega\}$ is inconsistent.
\end{defn}

\begin{defn} Suppose $\varphi(x;b)$ is a formula and $C$ is a set.
\begin{enumerate}
\item We say $\varphi(x;b)$ \emph{divides} over $C$ if it divides along some  $C$-indiscernible sequence in $\tp(b/C)$.
\item We say $\varphi(x;b)$ \emph{forks} over $C$ if there are formulas $(\psi_i(x;c_i))_{i<n}$ with $n< \omega$ such that $\varphi(x;b) \models \bigvee_{i < n} \psi_{i}(x;c_{i})$ and each $\psi_{i}(x;c_{i})$ divides over $C$. 
\item The notation $a \ind^{d}_{C} b$ means that $\mathrm{tp}(a/Cb)$ contains no formula that divides over $C$ and, similarly, $a \ind^{f}_{A} b$  means that $\mathrm{tp}(a/Cb)$ contains no formula that forks over $A$. 
\end{enumerate}
\end{defn}

We will be primarily concerned with extremely ``generic'' sequences, i.e., Morley sequences for global invariant types. 

\begin{defn}
A \emph{global partial type} $\pi(x)$ is a consistent set of formulas over $\monster$. 
A \emph{global type} is a global partial type which is complete. 
For a set $C$, we say the global partial type $\pi(x)$ is \emph{$C$-invariant} if, for all formulas $\varphi(x;y)$, if $b \equiv_{C} b'$, then $\varphi(x;b) \in \pi$ if and only if $\varphi(x;b') \in \pi$. 
\end{defn}

An important class of examples of global $C$-invariant types are the types that are finitely satisfiable in $C$. In any theory $T$, if $M \models T$, every type over $M$ has a global extension which is finitely satisfiable in $M$ (and therefore $M$-invariant). See Remark~\ref{rem:coheir} below.  

Over a general set $C$, there may be no global $C$-invariant types whatsoever. For this reason, when we want to work with invariant types (such as in the definition of Kim-dividing below), we usually work over a model.

\begin{defn}
Suppose $M \models T$.
\begin{enumerate}
\item We write $a \ind^{i}_{M} b$ if $\mathrm{tp}(a/Mb)$ extends to a global $M$-invariant type. 
\item We write $a \ind^{u}_{M} b$ if $\mathrm{tp}(a/Mb)$ extends to a global type finitely satisfiable in $M$.
\end{enumerate}
\end{defn}

\begin{rem}\label{rem:coheir}
The $u$ superscript comes from ``ultrafilter'', since global $M$-finitely satisfiable types all arise from the following construction:  if $p(x) \in S_{x}(M)$, then $\{\varphi(M) : \varphi(x) \in p\} \subseteq \Power(M^{x})$ generates a filter on $M^{x}$.  If $\D$ is an ultrafilter on $M^{x}$ extending this filter, then
\[
\Av(\D, \monster) = \{\varphi(x) \in L(\monster) : \varphi(M) \in \D\}
\]
is a global type extending $p$ which is finitely satisfiable in $M$. We write $\Av(\D,B)$ for $\Av(\D, \monster)$ restricted to parameters coming from $B$. 
\end{rem}

\begin{defn}
If $q$ is a global $C$-invariant type, then a \emph{Morley sequence}  over $C$ for $q$ is a sequence $(a_{i})_{i < \omega}$ such that $a_{i} \models q|_{Ca_{<i}}$ for all $i < \omega$. 
\end{defn}

\begin{fact}
By invariance, every Morley sequence over $C$ for $q$ is $C$-indiscernible. Furthermore, for a fixed global $C$-invariant type $q$ extending $\tp(b/C)$, if $\varphi(x;b)$ divides along some Morley sequence over $C$ for $q$, then it divides along every Morley sequence over $C$ for $q$.
\end{fact} 

\begin{defn} Suppose $\varphi(x;b)$ is a formula and $M$ is a model.
\begin{enumerate}
\item We say $\varphi(x;b)$ \emph{Kim-divides} over $M$ if it divides along a Morley sequence over $M$ for some global $M$-invariant type extending $\tp(b/M)$. 
\item We say $\varphi(x;b)$ \emph{Kim-forks} over $M$ if there are formulas $(\psi_i(x;c_i))_{i<n}$ with $n< \omega$ such that $\varphi(x;b) \models \bigvee_{i < n} \psi_{i}(x;c_{i})$ and each $\psi_{i}(x;c_{i})$ Kim-divides over $M$. 
\item The notation $a \ind^{Kd}_{M} b$ means that $\mathrm{tp}(a/Mb)$ contains no formula that Kim-divides over $M$ and, similarly, $a \ind^{K}_{M} b$  means that $\mathrm{tp}(a/Mb)$ contains no formula that Kim-forks over $M$. 
\end{enumerate}
\end{defn}

Kim-independence was introduced by Kaplan and the second-named author in~\cite{kaplan2017kim}, in the context of $\NSOP_1$ theories. They showed that if $T$ is $\NSOP_1$, then Kim-forking is equivalent to Kim-dividing, and $\ind^K$ satisfies many of the good properties of $\ind^f$ in simple theories. The definition of Kim-dividing was inspired by an earlier suggestion of Kim for studying independence in $\NTP_1$ theories \cite{KimNTP1}. 

\begin{rem}\label{rem:Kdparameters}
In a general theory, Kim-dividing as we have defined it is not always preserved under adding dummy parameters. That is, suppose $\varphi(x;y)$ is a formula, and write $\widehat{\varphi}(x;y,z)$ for the same formula consider in a larger variable context by appending unused variables $z$. It is possible that there are tuples $b$ and $c$ such that $\varphi(x;b)$ Kim-divides over $M$ but $\widehat{\varphi}(x;b,c)$ does not Kim-divide over $M$. The reason is that $\ind^i$ does not satisfy left-extension in general. More explicitly, if $q(y)$ is a global $M$-invariant type extending $\tp(b/M)$ (and witnessing the Kim-dividing of $\varphi(x;b)$), there may be no global $M$-invariant type $r(y,z)$ extending both $q(y)$ and $\tp(bc/M)$. Hanson has produced an explicit example of this behavior, see~\cite[Appendix C]{hanson2023biinvariant}.

As a result, we have to be careful about parameters when working with Kim-dividing in arbitrary theories. For example, if $\varphi(x;b)$ Kim-forks, we cannot assume in general that the witnessing Kim-dividing formulas $(\psi_i(x;c_i))_{i<n}$ have the same tuple of parameters. This will cause us some trouble in Section~\ref{sec:broom} below. 

All this suggests to us that our definition of Kim-dividing may not be the ``right'' one outside of the context of $\NSOP_1$ theories. If $T$ is $\NSOP_1$, then a formula Kim-divides over a model $M$ if and only if it Kim-divides along a coheir sequence over $M$ (a Morley sequence for a global type finitely satisfiable in $M$). And if Kim-dividing were defined as dividing along a coheir sequence, then the issue with dummy parameters would not arise, since $\ind^u$ always satisfies left-extension. However, focusing only on coheir sequences seems potentially too restrictive, and the definition of Kim-dividing in terms of invariant Morley sequences is well-established, so we retain it for this paper. 
\end{rem}

The diagram below depicts the implications between the notions of independence defined in this section. 
\[\xymatrix{
a\ind^u_M b \ar[r] & a \ind^i_M b \ar[r] & a\ind^f_M b \ar[r]\ar[d] & a\ind^K_M b\ar[d]\\
&& a\ind^d_M b\ar[r] & a\ind^{Kd}_M b
}\]

\begin{fact}[\cite{chernikov2012forking}, or \cite{adler2014kim}]\label{fact:ntp2}
In $\NTP_2$ theories, a formula $\varphi(x;b)$ divides over a model $M$ if and only if it Kim-divides over $M$. Further, forking and dividing coincide over models. So when $T$ is $\NTP_2$, $\ind^f_M = \ind^d_M = \ind^K_M = \ind^{Kd}_M$.
\end{fact}

It is a fact that simple theories are characterized by symmetry of $\ind^f$ \cite[Theorem 2.4]{kim2001simplicity}. So in a simple theory, if $p$ is a global $M$-invariant type and $a$ realizes $p|_{MB}$, then $B\ind^f_M a$ (since $a\ind^i_M B$ implies $a\ind^f_M B$ and $\ind^f$ is symmetric). Outside of the simple context, it can be useful to consider invariant types which always satisfy this instance of symmetry. These ``strict'' invariant types play an important role in Chernikov and Kaplan's analysis of forking in $\NTP_2$ theories~\cite{chernikov2012forking}. 

Similarly, $\NSOP_1$ theories are characterized by symmetry of $\ind^K$, so it makes sense in our context to consider ``Kim-strict'' invariant types, which are the analogue of strict invariant types for Kim-forking. 

\begin{defn}
Suppose $p \in S(\monster)$ is a global $M$-invariant type. 
\begin{enumerate}
\item We say $p$ is a \emph{strict invariant type} over $M$ when, for any set $B$, if $a \models p|_{MB}$, then $B \ind^{f}_{M} a$.  
\item We say $p$ is a \emph{Kim-strict invariant type} over $M$ when, for any set $B$, if $a \models p|_{MB}$, then $B \ind^{K}_{M} a$.  
\item A formula $\varphi(x;b)$ \emph{strictly divides} over $M$ if it divides along a Morley sequence for some global strictly $M$-invariant type extending $\tp(b/M)$. 
\item A formula $\varphi(x;b)$ \emph{Kim-strictly divides} over $M$ if it divides along a Morley sequence for some global Kim-strictly $M$-invariant type extending $\tp(b/M)$. 
\end{enumerate}
\end{defn}

Finally, for each of the variants of dividing defined above, we can also consider changing the quantifier from dividing along \emph{some} to dividing along \emph{every} indiscernible sequence of the appropriate kind. 

\begin{defn}
We say a formula $\varphi(x;b)$ \emph{universally Kim-divides}\footnote{Universal Kim-dividing is called ``strong Kim-dividing'' in~\cite{localchar} and ``Conant-dividing'' in~\cite{mutchnik2022conantindependence}.} over $M$ if it divides along Morley sequences for \emph{every} global $M$-invariant type extending $\tp(b/M)$. Similarly, we say $\varphi(x;b)$ \emph{universally strictly divides} over $M$ if it divides along Morley sequences for \emph{every} global strict $M$-invariant type extending $\tp(b/M)$, and we say $\varphi(x;b)$ \emph{universally Kim-strictly divides} over $M$ if it divides along Morley sequences for \emph{every} global Kim-strict $M$-invariant type extending $\tp(b/M)$. 
\end{defn}

\begin{rem}
For completeness, we could say a formula $\varphi(x;b)$ \emph{universally divides} over $C$ if it divides along \emph{every} $C$-indiscernible sequence in $\tp(b/C)$. Note, however, that since the constant sequence with $b_i = b$ for all $i$ is $C$-indiscernible, a universally dividing formula is inconsistent.  
\end{rem}

\subsection{The Broom Lemma}\label{sec:broom}

It is clear that universal Kim-dividing implies Kim-dividing, since every type over a model $M$ extends to a global $M$-invariant type (see Remark~\ref{rem:coheir}). However, it is not so clear that universal (Kim-)strict dividing implies (Kim-)strict dividing. 

In~\cite{chernikov2012forking}, Chernikov and Kaplan proved that in an $\NTP_2$ theory, every type over a model $M$ extends to a global strictly $M$-invariant type, using a device they called the Broom Lemma. It turns out that their argument applies to all theories, if we replace strict invariance with Kim-strict invariance. 

A key step in the Chernikov--Kaplan argument is that forking implies quasi-dividing in the sense of the following definition. 

\begin{defn}
A formula $\varphi(x;b)$ \emph{quasi-divides} over $M$ if the conjunction of finitely many conjugates of $\varphi(x;b)$ over $M$ is inconsistent. That is, if there exist $(b_i)_{i<k}$ with $k< \omega$ and $b_i\equiv_M b$ for all $i<k$ such that $\bigwedge_{i<k}\varphi(x;b_i)$ is inconsistent. 
\end{defn}

\begin{rem}
We could say that $\varphi(x;b)$ \emph{quasi-forks} over $M$ if there are formulas $(\psi_i(x;c_i))_{i<n}$ with $n< \omega$ such that $\varphi(x;b) \models \bigvee_{i < n} \psi_{i}(x;c_{i})$ and each $\psi_{i}(x;c_{i})$ quasi-divides over $M$. It is worth noting that $a\ind^i_M b$ if and only if $\tp(a/Mb)$ contains no formula which quasi-forks. But we will not make use of this fact.
\end{rem}

The original Broom Lemma argument from~\cite{chernikov2012forking} does not appear to generalize directly to our context. But in~\cite{adler2014kim}, Adler used a variant of the Broom Lemma, which he called the Vacuum Cleaner Lemma, to give a simplified proof of some of the Chernikov--Kaplan results on $\NTP_2$ theories. Adler's proof (\cite{adler2014kim}, Lemma 3) goes through verbatim to prove the following result, in the context of an arbitrary theory $T$. 

\begin{lem}[Vacuum Cleaner for Kim-dividing\footnote{A similar modified Broom Lemma played a key role in Mutchnik's proof of the equivalence of $\NSOP_1$ and $\NSOP_2$ in~\cite{mutchnik2022nsop2}.}] \label{vacuum cleaner}
Let $\pi(x)$ be an $M$-invariant partial type and suppose 
\[
\pi(x) \models \psi(x;b) \vee \bigvee_{i < n} \varphi_{i}(x;c),
\]
where $b \ind^{i}_{M} c$ and each $\varphi_{i}(x;c)$ Kim-divides over $M$.  Then $\pi(x) \models \psi(x;b)$.  
\end{lem}

\begin{cor} \label{inconsistent}
Suppose $\theta(x;b)\models \bigvee_{i < n} \varphi_{i}(x;c)$, where each $\varphi_{i}(x;c)$ Kim-divides over $M$. Then $\theta(x;b)$ quasi-divides over $M$. 
\end{cor}

\begin{proof}
Let $\pi(x) = \{\theta(x;b') : b'\equiv_M b\}$ and let $\psi$ be $\bot$. By Lemma~\ref{vacuum cleaner}, $\pi(x)$ is inconsistent, so, by compactness, $\theta(x;b)$ quasi-divides. 
\end{proof}

Corollary~\ref{inconsistent} seems to say that Kim-forking formulas quasi-divide. But, as noted in Remark~\ref{rem:Kdparameters} above, we cannot assume in general that in the finite disjunction $\bigvee_{i<n}\varphi_i(x;c_i)$ witnessing Kim-forking, all of the Kim-dividing formulas have the same tuple of parameters $c$. Unfortunately, this assumption seems crucial in Adler's proof of the Vacuum Cleaner Lemma. As in Remark~\ref{rem:Kdparameters}, this would not be an issue if we defined Kim-dividing in terms of dividing along coheir sequences. 

Nevertheless, it is true in general that Kim-forking formulas quasi-divide. We present an alternative proof, based on an idea due to Hanson. 

\begin{lem} \label{lem:quasidividing}
Let $\varphi(x;b)$ be a formula. Suppose that the conjunction of finitely many conjugates of $\varphi(x;b)$ over $M$ entails a formula which quasi-divides over $M$. Then $\varphi(x;b)$ quasi-divides over $M$. 
\end{lem}
\begin{proof}
By hypothesis, there exist $(b_i)_{i<k}$ with $b_i\equiv_M b$ for all $i<k$ such that $\bigwedge_{i<k}\varphi(x;b_i)\models \psi(x;c)$, and $\psi(x;c)$ quasi-divides over $M$. Then there exist $(c_j)_{j<n}$ with $c_j\equiv_M c$ for all $j<n$ such that $\bigwedge_{j<n}\psi(x;c_j)$ is inconsistent. 

For each $j<n$, pick $(b_{i,j})_{i<k}$ such that $b_{0,j} \dots b_{(k-1),j} c_j \equiv_M b_0\dots b_{k-1} c$. Then \[\bigwedge_{j<n}\bigwedge_{i<k}\varphi(x;b_{i,j})\models \bigwedge_{j<n} \psi(x;c_j).\]
For all $i<k$ and $j<n$, $b_{i,j}\equiv_M b_i \equiv_M b$, so this is a finite conjunction of conjugates of $\varphi(x;b)$ over $M$ which is inconsistent. 
\end{proof}

\begin{lem} \label{lem:dustpan}
Suppose $\varphi(x;b)$ Kim-divides over $M$. Then for any $(b_i)_{i<\ell}$ such that $b_i\equiv_M b$ for all $i<\ell$, $\bigvee_{i<\ell}\varphi(x;b_i)$ quasi-divides over $M$. 
\end{lem}
\begin{proof}
Write $\Phi(x;\overline{b})$ for the formula $\bigvee_{i<\ell}\varphi(x;b_i)$. Our goal is to show that $\Phi(x;\overline{b})$ quasi-divides. Let $q(y)$ be a global $M$-invariant type extending $\tp(b/M)$ and witnessing that $\varphi(x;b)$ Kim-divides over $M$. Let $k$ be such that, if $(b'_i)_{i< \omega}$ is a Morley sequence for $q$ over $M$, $\{\varphi(x;b'_i) : i<k\}$ is inconsistent.  

Write $\ell^{\leq m}_*$ for the set of functions $\eta\colon n\to \ell$ with $0<n\leq m$, that is, $\ell^{\leq m}_{*} = \ell^{\leq m} \setminus \{\langle \rangle\}$. We will prove by induction that for all $m\leq k$, we can find $(b_\eta)_{\eta\in \ell^{\leq m}_*}$ such that 
\begin{enumerate}[(1)]
\item For each $\rho\in \ell^m$, $(b_\rho,b_{\rho|_{m-1}},\dots,b_{\rho|_1})$ begins a Morley sequence in $q$ over $M$. 
\item For each $\eta\in \ell^{<m}$, writing $\overline{b}'_\eta$ for the tuple $(b_{\eta^\frown\langle i\rangle})_{i<\ell}$, we have $\overline{b}'_\eta\equiv_M \overline{b}$. 
\end{enumerate}

In the base case, when $m = 0$, $\ell^{\leq m}_*$ is empty, and the conditions are satisfied vacuously. 

For the inductive step, suppose we are given $F_0 = (b_\eta)_{\eta\in \ell^{\leq m}_*}$ satisfying the conditions, with $m<k$. Let $b''_0$ realize $q|_{MF_0}$. By condition (1), we now have that for each $\rho\in \ell^m$, $(b_\rho,b_{\rho|_{m-1}},\dots,b_{\rho|_1},b''_0)$ begins a Morley sequence in $q$ over $M$.

Since $b''_0\equiv_M b_0$, we can pick $(b''_i)_{0<i<\ell}$ so that $(b''_i)_{i<\ell} \equiv_M \overline{b}$. Now, for each $0< i < \ell$, pick $F_i$ so that $F_i b''_i\equiv_M F_0 b''_0$. Reindex so that we have a forest indexed by $\ell^{\leq (m+1)}_*$, with $(b''_i)_{i<\ell}$ as the ``bottom layer''  $\overline{b}'_{\langle\rangle}$. This completes the inductive construction. 

Now we have $(b_\eta)_{\eta\in \ell^{\leq k}_*}$ satisfying (1) and (2). Observe that \[\bigwedge_{\eta\in \ell^{<k}} \bigvee_{i<\ell} \varphi(x;b_{\eta^\frown\langle i\rangle})\models \bigvee_{\rho\in \ell^{k}} \bigwedge_{1\leq i\leq k} \varphi(x;b_{\rho|_i}).\] By (1), for each $\rho\in \ell^k$, $\bigwedge_{1\leq i\leq k} \varphi(x;b_{\rho|_i})$ is inconsistent. Thus the left-hand side, which is $\bigwedge_{\eta\in \ell^{<k}} \Phi(x;\overline{b}'_\eta)$, is inconsistent. By (2), this shows that $\Phi(x;\overline{b})$ quasi-divides over $M$.  
\end{proof}

\begin{lem} \label{lem:mop}
Suppose $(\varphi_i(x;b_i))_{i<n}$ are formulas, each of which Kim-divides over $M$. For each $i<n$, let $\theta_i(x;c_i)$ be a disjunction of finitely many conjugates of $\varphi_i(x;b_i)$. Then $\bigvee_{i<n} \theta_i(x;c_i)$ quasi-divides over $M$. 
\end{lem}
\begin{proof}
By induction on $n$. When $n = 0$, the disjunction is $\bot$, which quasi-divides over $M$. For the inductive step, we consider $\bigvee_{i<n+1} \theta_i(x;c_i)$. Now $\theta_n(x;c_n)$ is a disjunction of finitely many conjugates of $\varphi_n(x;b_n)$. By Lemma~\ref{lem:dustpan}, $\theta_n(x;c_n)$ quasi-divides over $M$, so there are $(c_{nj})_{j<k}$ with $c_{nj}\equiv_M c_n$ for all $j<k$ such that $\bigwedge_{j<k}\theta_n(x;c_{nj})$ is inconsistent. 

For each $j<k$, pick $(c_{ij})_{i<n}$ such that $c_{0j}\dots c_{nj}\equiv_M c_0\dots c_n$. Consider the conjunction \[\bigwedge_{j<k} \bigvee_{i<n+1}\theta_i(x;c_{ij}).\] 
Whenever this formula is true, there must be some $j<k$ such that some disjunct $\theta_i(x;c_{ij})$ with $i\neq n$ is true, since $\bigwedge_{j<k}\theta_n(x;c_{nj})$ is inconsistent. Thus  
\[\bigwedge_{j<k} \bigvee_{i<n+1}\theta_i(x;c_{ij})\models \bigvee_{i<n}\bigvee_{j<k}\theta_i(x;c_{ij}).\]
Since each formula $\bigvee_{j<k}\theta_i(x;c_{ij})$ is a disjunction of finitely many conjugates of $\varphi_i(x;b_i)$, by induction $\bigvee_{i<n}\bigvee_{j<k}\theta_i(x;c_{ij})$ quasi-divides over $M$. By Lemma~\ref{lem:quasidividing}, $\bigvee_{i<n+1}\theta_i(x;c_{ij})$ quasi-divides over $M$. 
\end{proof}

\begin{cor}\label{cor:kfqd}
Every formula which Kim-forks over $M$ quasi-divides over $M$. 
\end{cor}
\begin{proof}
Suppose $\varphi(x;b)$ Kim-forks over $M$. Then $\varphi(x;b)\models \bigvee_{i<n} \psi_i(x;c_i)$ such that each $\psi_i(x;c_i)$ Kim-divides over $M$. By Lemma~\ref{lem:mop} (taking each $\theta_i$ to be $\psi_i(x;c_i)$), $\bigvee_{i<n} \psi(x;c_i)$ quasi-divides over $M$, and hence so does $\varphi(x;b)$ by Lemma~\ref{lem:quasidividing}.
\end{proof}

\begin{thm} \label{thm:kimstrict}
Every type over $M \models T$ has a Kim-strict $M$-invariant global extension.
\end{thm}

\begin{proof}
Given $p(x) = \text{tp}(a/M)$, consider the following collection of formulas:
\[
p(x) \cup \{\psi(x;c) \leftrightarrow \psi(x;c') : c \equiv_{M} c'\}\cup \{\neg \varphi(x;b) : \varphi(a;y) \text{ Kim-forks over }M\}.
\]
We must show that this is a consistent partial type.  Suppose not; then, by compactness, 
\[
p(x) \cup \{\psi(x;c) \leftrightarrow \psi(x;c') : c \equiv_{M} c'\} \models \varphi(x;b),
\]
for some formula $\varphi(x;y)$ such that $\varphi(a;y)$ Kim-forks over $M$.  

By Corollary~\ref{cor:kfqd}, there are $(a_i)_{i<m}$ such that $a_i\equiv_M a$ for all $i<m$ and $\{\varphi(a_{i},y) : i < m\}$ is inconsistent.  Let $r(x_0,\dots,x_{m-1})$ be a global $M$-invariant type extending $\text{tp}(a_{0},\ldots, a_{m-1}/M)$, and for $j<m$, let $r(x_j)$ be the restriction of $r$ to formulas with free variables from $x_j$. Then each $r(x_j)$ is a global $M$-invariant type extending $p(x_j)$, so $r(x_j)\models \varphi(x_j,b)$. Thus,
\[
r(x_{0},\ldots,x_{m-1}) \models \bigwedge_{j<m} \varphi(x_{j};b),
\]
and therefore $\exists y\,\bigwedge_{j < m} \varphi(x_{j},y)\in r$. This contradicts the fact that $r$ extends $\text{tp}(a_{0},\ldots, a_{m-1}/M)$.  
\end{proof}

\begin{cor}
If $\varphi(x;b)$ universally Kim-strictly divides over $M$, then it Kim-strictly divides over $M$. 
\end{cor}
\begin{proof}
By Theorem~\ref{thm:kimstrict}, $\tp(b/M)$ has a Kim-strict $M$-invariant global extension $q(y)$. Since $\varphi(x;b)$ universally Kim-strictly divides over $M$, it divides along Morley sequences for $q$. Thus, it Kim-strictly divides over $M$. 
\end{proof}

Note that the only properties of Kim-forking used in the proof of Theorem~\ref{thm:kimstrict} are (a) that the Kim-forking formulas form an ideal (i.e., they are closed under finite disjunctions), and (b) that every Kim-forking formula quasi-divides. In unpublished work, Hanson has shown that there is a largest $M$-invariant ideal which contains only quasi-dividing formulas, called the ``fracturing'' ideal. The proof of Theorem~\ref{thm:kimstrict} works just as well to show that $\tp(a/M)$ extends to a global $M$-invariant extension containing no formula $\varphi(x;b)$ such that $\varphi(a;y)$ fractures.

\begin{rem}\label{rem:nostrict}
In~\cite[Subsection 5.1]{chernikov2012forking}, Chernikov and Kaplan present an example, due to Martin Ziegler, of a theory $T$ in which there is a model $M \models T$ and a type over $M$ with no global extension that is strict invariant over $M$. This shows that, in general, Theorem~\ref{thm:kimstrict} cannot be improved to establish the existence of global strict invariant types over models in arbitrary theories. 
\end{rem}

We conclude this section with a diagram showing the implications between the various notions of dividing (over models) introduced in Section~\ref{sec:dividing}. All implications hold in an arbitrary theory, except for the implication from universally strictly divides to strictly divides, which requires $\NTP_2$. 

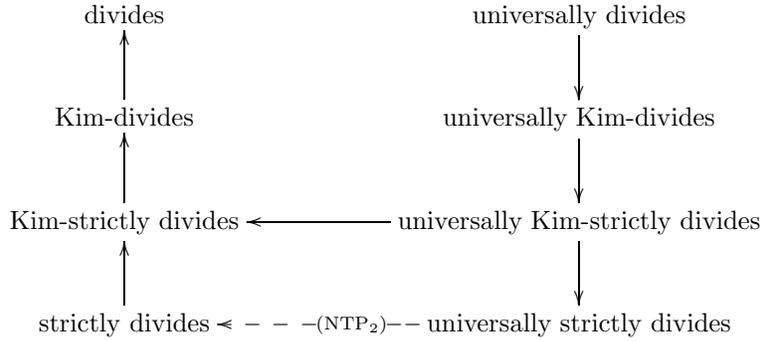
\[
\xymatrix{
\text{divides} 
 && \text{universally divides}\ar[d]\\
\text{Kim-divides} \ar[u]  && \text{universally Kim-divides}\ar[d]\\
\text{Kim-strictly divides} \ar[u] 
&& \text{universally Kim-strictly divides}\ar[ll]\ar[d]\\
\text{strictly divides} \ar[u] && \text{universally strictly divides} \ar@{-->}|{(\NTP_2)}[ll]
}\]

\section{A Diversity of Kim's Lemmas}\label{sec:diversity}

In this section, we survey the characterizations of simplicity, $\NSOP_1$, and $\NTP_2$ by Kim's Lemmas, and we introduce our New Kim's Lemma. We begin with the original Kim's Lemma in the context of simple theories. 

\begin{thm}[{\cite[Proposition 2.1]{kim1998forking}}, {\cite[Theorem 2.4]{kim2001simplicity}}]
The following are equivalent: 
\begin{enumerate}
\item $T$ is simple.
\item For all sets $C$, if a formula $\varphi(x;b)$ divides over $C$, then it divides along every $\ind^f$-Morley sequence over $C$. 
\end{enumerate}
\end{thm}

In this paper, we are primarily interested in Morley sequences for global invariant types over models (rather than $\ind^f$-Morley sequences over arbitrary sets), so we are led to consider the following variant of (2): 

\begin{enumerate}
\setcounter{enumi}{2}
\item For all models $M$, if a formula $\varphi(x;b)$ divides over $M$, then it universally Kim-divides over $M$. 
\end{enumerate}

Note that (3) is a weakening of (2), since it restricts to the special case of models, and since every Morley sequence for a global $M$-invariant type is a $\ind^f$-Morley sequence over $M$. But (3) is still strong enough to characterize simplicity. 

The equivalence of (1) and (3) has not (to our knowledge) appeared explicitly in the literature, but it does follow directly from facts in the literature. We have already observed that (1) implies (2) and (2) implies (3). Conversely, (3) implies, in particular, that Kim-dividing implies universal Kim-dividing, so $T$ is $\NSOP_{1}$ (by the Kim's Lemma for $\NSOP_1$ theories, Theorem~\ref{thm:kimnsop1} below). Thus $T$ is an $\NSOP_{1}$ theory in which dividing and Kim-dividing coincide over models, so $T$ is simple by~\cite[Proposition 8.4]{kaplan2017kim}.

For the reader's convenience, and to give an indication of the typical flavor of arguments relating variants of Kim's Lemma to combinatorial configurations like the tree property, we will also give a self-contained proof of the equivalence of (1) and (3). 

\begin{proof} (1)$\implies$(3).  Suppose $(3)$ fails, so there is a model $M \models T$, a formula $\varphi(x;b)$ that divides over $M$, and a global $M$-invariant type $q \supseteq \mathrm{tp}(b/M)$ such that $\varphi(x;b)$ does not divide along Morley sequences over $M$ for $q$. Let $(b_{i})_{i < \omega}$ be an $M$-indiscernible sequence in $\tp(b/M)$ such that $\{\varphi(x;b_{i}) : i < \omega\}$ is inconsistent (and hence $k$-inconsistent for some $k$). By induction, we will build for each $n < \omega$, a tree $(c_{\eta})_{\eta \in \omega^{\leq n}}$ satisfying the following:
\begin{itemize}
\item For all $\eta \in \omega^{< n}$, $(c_{\eta^{\frown}\langle i \rangle})_{i < \omega} \equiv_{M} (b_{i})_{i < \omega}$.
\item For all $\nu \in \omega^{n}$, $(c_{\nu}, c_{\nu|(n-1)}, \ldots, c_{\nu|0})$ begins a Morley sequence in $q$ over $M$. 
\end{itemize} 

For $n = 0$, we define $c_{\langle\rangle} = b$. The conditions are trivially satisfied. 

For the inductive step, we are given a tree $(c_{\eta,0})_{\eta \in \omega^{\leq n}}$. Since $c_{\langle\rangle,0}\models q|_M$, we have $c_{\langle\rangle,0} \equiv_{M} b$, and there is a sequence $(c_{\langle\rangle,i})_{i < \omega}$ beginning with $c_{\langle\rangle,0}$ such that $(c_{\langle\rangle,i})_{i < \omega} \equiv_{M} (b_{i})_{i < \omega}$. For each $i$, $c_{\langle\rangle,i} \equiv_{M} c_{\langle\rangle,0}$, so we can choose a tree $(c_{\eta,i})_{\eta \in \omega^{\leq n}}$ with root $c_{\langle\rangle,i}$ such that $(c_{\eta,0})_{\eta \in \omega^{\leq n}} \equiv_{M} (c_{\eta,i})_{\eta \in \omega^{\leq n}}$.  Let $c_{\langle\rangle}$ be a realization of $q|_{M\{c_{\eta,i} : \eta \in \omega^{\leq n},i < \omega \}}$.  Then we reindex to define a tree $(c_{\eta})_{\eta \in \omega^{\leq n+1}}$ by setting $c_{\langle i \rangle^{\frown} \eta} = c_{\eta,i}$ for all $i < \omega$ and $\eta \in \omega^{\leq n}$.   

Note that, for each $n$, the tree $(c_{\eta})_{\eta \in \omega^{\leq n}}$ that we constructed has the following properties.  First, for each $\eta \in \omega^{<n}$, $\{\varphi(x;c_{\eta ^\frown \langle i \rangle}) : i < \omega\}$ is $k$-inconsistent, by the first bullet point above.  Secondly, for all $\nu \in \omega^{n}$, $\{\varphi(x;c_{\nu | \ell}) : \ell \leq n\}$ is consistent, by the second bullet point and our assumption on $q$. By compactness, $\varphi(x;y)$ has TP, and $T$ is not simple. 

\medskip

(3)$\implies$(1) Suppose $T$ has TP witnessed by $\varphi(x;y)$, $k< \omega$, and $(a_{\eta})_{\eta \in \omega^{<\omega}}$. Fix a Skolemization $T^{\mathrm{Sk}}$ of $T$. The same data shows that $\varphi(x;y)$ has TP modulo $T^{\mathrm{Sk}}$. 

By compactness, we can obtain a tree $(a_{\eta})_{\eta \in \kappa^{<\omega}}$, where $\kappa>2^{|T|}$, and which satisfies the obvious extensions of the defining conditions of the tree property. 

We build an array $(b_{i,j})_{i,j<\omega}$ and $\rho\in \kappa^\omega$ with the following properties (in $T^{\mathrm{Sk}}$): 
\begin{itemize}
\item $b_{i,0} = a_{\rho|(i+1)}$ for all $i<\omega$ (and therefore $\{\varphi(x;b_{i,0}): i<\omega\}$ is consistent). 
\item For all $i<\omega$, $\{\varphi(x;b_{i,j}):j<\omega\}$ is $k$-inconsistent. 
\item For all $i<\omega$, $(b_{i,j})_{j<\omega}$ is indiscernible over $(b_{\ell,0})_{\ell<i}$.
\end{itemize}

We proceed by recursion on $i$. Given $\rho|n$ and $(b_{i,j})_{i<n,j<\omega}$, let $\eta = \rho|n$, and consider the sequence $(a_{\eta^\frown\langle \alpha\rangle})_{\alpha<\kappa}$. By the conditions on $\kappa$, we can find a subsequence $I = (a_{\eta^\frown \langle \alpha_j\rangle})_{j<\omega}$ such that each $a_{\eta^\frown \langle \alpha_j\rangle}$ satisfies the same complete type $p(y)$ over $(b_{i,0})_{i<n}$. Let $(b_{n,j})_{j<\omega}$ be a sequence which is indiscernible and locally based on $I$ over $(b_{i,0})_{i<n}$ (i.e. realizes the Ehrenfeucht-Mostowski type of $I$ over $(b_{i,0})_{i < n}$). It follows that each $b_{n,j}$ satisfies $p(y)$, so we can assume that $b_{n,0} = a_{\eta^\frown\langle \alpha_0\rangle}$ and let $\rho(n) = \alpha_0$. It also follows that $\{\varphi(x;b_{n,j}):j<\omega\}$ is $k$-inconsistent. This completes the construction. 

For each $i < \omega$, let $\overline{b}_{i} = (b_{i,j})_{j<\omega}$, and let $J = (\overline{b}_i)_{i<\omega}$. Let $J' = (\overline{b}'_{i})_{i < \omega + \omega}$ be a sequence which is indiscernible and locally based on $J$ (over $\emptyset$). Writing each $\overline{b}'_i$ as $(b'_{i,j})_{j<\omega}$, we retain consistency of $\{\varphi(x;b'_{i,0}): i<\omega+\omega\}$, $k$-inconsistency of $\{\varphi(x;b'_{i,j}):j<\omega\}$ for all $i<\omega+\omega$, and indiscernibility of $(b'_{i,j})_{j<\omega}$ over $(b'_{\ell,0})_{\ell<i}$ for all $i<\omega+\omega$.

Let $M$ be the Skolem hull of $(b'_{i,0})_{i<\omega}$. By indiscernibility, $\tp(b'_{\omega,0}/M(b'_{i,0})_{i>\omega})$ is finitely satisfiable in $M$ and therefore extends to a global $M$-finitely satisfiable (and therefore $M$-invariant) type $q$. Moreover, by indiscernibility, $b'_{\omega + i,0} \models q|_{M(b'_{n,0})_{n>\omega+i}}$ for all $i$, which shows that for all $n$, $(b'_{\omega+n,0},\dots,b'_{\omega,0})$ begins a Morley sequence for $q$ over $M$. By construction, $\{\varphi(x;b'_{\omega + i,0}) : i < \omega\}$ is consistent, so $\varphi(x;b'_{\omega,0})$ does not divide along Morley sequences for $q$ over $M$. However, $\varphi(x;b'_{\omega,0})$ does divide along the $M$-indiscernible sequence $\overline{b}'_{\omega}$.

Taking the reduct back to $T$, the restriction $q|_L$ of $q$ to $L$-formulas is still finitely satisfiable in $M$, Morley sequences in $q$ are also Morley sequences in $q|_L$, and the $M$-indiscernible sequence $\overline{b}'_{\omega}$ remains $M$-indiscernible in the reduct. Thus $\varphi(x;b'_{\omega,0})$ divides but does not universally Kim-divide with respect to $T$, and (3) fails.  
\end{proof}

\begin{exmp}\label{ex:eqrel}
Let $T_E$ be the theory of an equivalence relation $E$ with infinitely many classes, each of which is infinite. $T_E$ is a simple theory (in fact, it is $\omega$-stable). Let $M\models T_E$, and let $b$ be an element of $\monster$ in an equivalence class which is not represented in $M$. There are three types of $M$-indiscernible sequence $(b_i)_{i< \omega}$ in $\tp(b/M)$: (a) constant sequences, in which $b_i = b_j$ for all $i,j< \omega$, (b) sequences contained in one equivalence class, in which $b_i\neq b_j$ but $b_i E b_j$ for all $i\neq j$, and (c) sequences that move across equivalence classes, in which $\lnot b_i E b_j$ for all $i\neq j$. 

The formula $xEb$ divides along sequences of type (c), but not along sequences of type (a) or (b). Is there a general explanation for this behavior? Kim's Lemma gives the answer: the dividing formula $xEb$ universally Kim-divides, and every Morley sequence for a global $M$-invariant type extending $\tp(b/M)$ has type (c). 

Indeed, if $q(y)$ is a global $M$-invariant type extending $\tp(b/M)$, we will show that $q$ cannot contain the formula $yEc$ for any $c\in \monster$. If $cEm$ for some $m\in M$, then since $\lnot yEm\in \tp(b/M)$, $\lnot yEc\in q$. And if the equivalence class of $c$ is not represented in $M$, then letting $c'$ be another element inequivalent to $c$ whose equivalence class is not represented in $M$, $q$ cannot contain both $yEc$ and $yEc'$, but $\tp(c/M) = \tp(c'/M)$, so by invariance $q$ does not contain $yEc$. It follows that a Morley sequence for $q$ has type (c). 
\end{exmp}

Next, we turn to the Kim's Lemma characterization of $\NSOP_1$ theories. 

\begin{thm} 
[{\cite[Theorem 3.16]{kaplan2017kim}}]\label{thm:kimnsop1}
The following are equivalent: 
\begin{enumerate}
\item $T$ is $\NSOP_{1}$. 
\item For all models $M$, if a formula $\varphi(x;b)$ Kim-divides over $M$, then it universally Kim-divides over $M$. 
\end{enumerate}
\end{thm}

\begin{exmp}\label{ex:feq}
$T^*_{\feq}$, the generic theory of parametrized equivalence relations, is $\NSOP_1$ and has $\TP_2$. It is the complete theory of the Fra\"iss\'e limit of the Fra\"iss\'e class $\K_{\feq}$. The language has two sorts, $O$ and $P$, and one ternary relation $y E_{x} z$, where the subscript $x$ has type $P$ and $y$ and $z$ have type $O$. A finite structure $A$ is in $\K_{\feq}$ if for all $a\in P(A)$, $E_a$ defines an equivalence relation on $O(A)$.  

Let $M\models T^*_{\feq}$, let $c\in P(\monster)\setminus P(M)$, and let $b\in O(\monster)$ such that the $E_c$-class of $b$ is not represented in $O(M)$. The formula $xE_c b$ divides over $M$, along any $M$-indiscernible sequence $(b_i,c_i)_{i< \omega}$ such that $c_i = c$ for all $i$ and $\lnot b_i E_c b_j$ for all $i\neq j$. But if $p(y,z)$ is a global $M$-invariant type extending $\tp(bc/M)$ and $I = (b_i,c_i)_{i< \omega}$ is a Morley sequence for $p$, then $c_i\neq c_j$ for all $i\neq j$, and $x E_c b$ does not divide along $I$. Indeed, by compactness and the genericity properties of the Fra\"iss\'e limit, if $(c_i)_{i< \omega}$ is any sequence of pairwise distinct elements of $P(\monster)$, and $C_i$ is an $E_{c_i}$ class for each $i< \omega$, then we can find $a\in O(\monster)$ such that $a\in C_i$ for all $i<\omega$. It follows that $x E_c b$ does not Kim-divide, and hence does not universally Kim-divide, so the Kim's Lemma for simple theories fails in $T^*_\feq$. 

Now let $m\in P(M)$, and let $b'\in O(\monster)$ such that the $E_m$-class of $b'$ is not represented in $O(M)$. Then the formula $x E_m b'$ Kim-divides over $M$, and, as predicted by the Kim's Lemma for $\NSOP_1$ theories, it universally Kim-divides over $M$. Indeed, if $p(y,z)$ is any global $M$-invariant type extending $\tp(b'm/M)$, and $I = (b_i,m_i)_{i< \omega}$ is a Morley sequence for $p$, then $m_i = m$ for all $i<\omega$ and $(b_i)_{i< \omega}$ is an indiscernible sequence of type (c) for $E_m$, according to the terminology in Example~\ref{ex:eqrel}. Thus $x E_m b'$ divides along $I$. 
\end{exmp}

Finally, we turn to the Kim's Lemma characterization of $\NTP_2$ theories. 

\begin{thm}
[{\cite[Lemma 3.14]{chernikov2012forking}}, {\cite[Theorem 4.9]{ChernikovNTP2}}]\label{thm:ntp2kim}
The following are equivalent: 
\begin{enumerate}
\item $T$ is $\NTP_2$. 
\item For all models $M$, if a formula $\varphi(x;b)$ divides over $M$, then it universally Kim-strictly divides over $M$. 
\end{enumerate} 
\end{thm}

Note that the notion of Kim-strict dividing does not appear in~\cite{chernikov2012forking} or~\cite{ChernikovNTP2}. Instead, Chernikov and Kaplan prove that (1) is equivalent to (3): 
\begin{enumerate}
\setcounter{enumi}{2}
\item For all models $M$, if a formula $\varphi(x;b)$ divides over $M$, then it universally  strictly divides over $M$. 
\end{enumerate}
But since Kim-strict invariant types coincide with strict invariant types in $\NTP_2$ theories (by Fact~\ref{fact:ntp2}), and universal Kim-strict dividing implies universal strict dividing in arbitrary theories, it follows immediately that (1), (2), and (3) are all equivalent. We have chosen to focus on Kim-strict dividing because it behaves better outside of the $\NTP_2$ context (by Theorem~\ref{thm:kimstrict} and Remark~\ref{rem:nostrict}). 

\begin{exmp}\label{ex:dlo}
$\DLO$, the theory of dense linear orders without endpoints, is $\NTP_2$ (in fact, it is $\NIP$) and has $\SOP_1$ (in fact, it has $\SOP$). Let $M\models \DLO$, and $b<c$ be two elements in $\monster\setminus M$ living in the same cut in $M$ (so there is no $m\in M$ with $b<m<c$). Now $q(y,z) = \tp(bc/M)$ has three global $M$-invariant extensions. By quantifier elimination, each is determined by the order relations between $y$ and $z$ and the elements $d\in \monster$ living in the same cut in $M$ as $b$ and $c$.  
\begin{enumerate}
\item Let $q_1$ be the global type containing $d<y<z$ for all such $d$. A Morley sequence $(b_i,c_i)_{i< \omega}$ for $q_1$ has $b_0 < c_0 < b_1 < c_1 < b_2 < c_2 < \cdots$. 
\item Let $q_2$ be the global type containing $y<z<d$ for all such $d$. A Morley sequence $(b_i,c_i)_{i< \omega}$ for $q_2$ has $\cdots <b_2 < c_2 < b_1 < c_1 < b_0 < c_0$. 
\item Let $q_3$ be the global type containing $y<d<z$ for all such $d$. A Morley sequence $(b_i,c_i)_{i< \omega}$ for $q_3$ has $\cdots < b_2 < b_1 < b_0 < c_0 < c_1 < c_2 < \cdots$.
\end{enumerate}

The formula $b < x < c$ divides along Morley sequences for $q_1$ and $q_2$, but not along Morley sequences for $q_3$. This shows that the Kim's Lemma for $\NSOP_1$ theories fails in $\DLO$: Kim-dividing does not imply universal Kim-dividing. But the Kim's Lemma for $\NTP_2$ theories explains which Morley sequences we should expect a dividing formula to divide along. Indeed, the dividing formula $b < x < c$ universally Kim-strictly divides, and we will show that $q_1$ and $q_2$ are Kim-strict, while $q_3$ is not.

Suppose $A\subseteq \monster$, and suppose $b'c'\models q_i|_{MA}$ for some $i\in \{1,2,3\}$. If $i = 1$ or $2$, then there is no $a\in A$ such that $b < a < c$, and it follows that $A\ind^{K}_M b'c'$. So $q_1$ and $q_2$ are Kim-strict. 

On the other hand, if $i = 3$, and if $A$ contains an element $a$ living in the same cut in $M$ and $b$ and $c$, then $b' < a < c'$. Thus $\tp(A/Mb'c')$ contains the Kim-dividing formula $b' < x < c'$, and $A\nind^{K}_M b'c'$. So $q_3$ is not Kim-strict. 
\end{exmp}

We can now fill in the diagram from the end of Section~\ref{sec:preliminaries} with the implications coming from the variants of Kim's Lemma which hold in various contexts, as well as our New Kim's Lemma. 
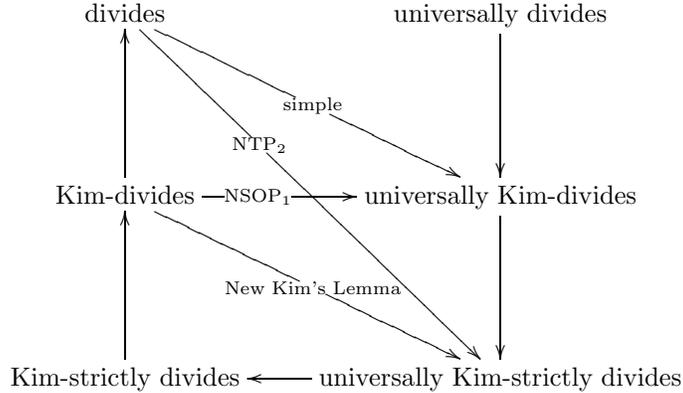
\[
\xymatrix{
\text{divides} 
\ar|{\text{simple}}[ddr] \ar|-<<<<<<<<<<<<<<<<<<<<<{\NTP_2}[ddddr] & \text{universally divides}\ar[dd]\\\\
\text{Kim-divides} \ar[uu] \ar|<<<<<<<<{\NSOP_1}[r]\ar|{\text{New Kim's Lemma}}[ddr] & \text{universally Kim-divides}\ar[dd]\\\\
\text{Kim-strictly divides} \ar[uu] 
& \text{universally Kim-strictly divides}\ar[l]}
\]

\begin{defn}\label{def:nkl}
$T$ satisfies \emph{New Kim's Lemma} if for all models $M$, if a formula $\varphi(x;b)$ Kim-divides over $M$, then it universally Kim-strictly divides over $M$. 
\end{defn}

We will give some examples and non-examples of New Kim's Lemma in the next section. For now, let us observe a simple consequence. Variants of Kim's Lemma allow us to prove that the relevant notions of forking and dividing coincide, and the usual proof works here as well. 

\begin{prop}
Suppose $T$ satisfies New Kim's Lemma and $M\models T$. Then a formula $\varphi(x;b)$ Kim-forks over $M$ if and only if it Kim-divides over $M$.
\end{prop}
\begin{proof}
Kim-dividing implies Kim-forking by definition. So suppose $\varphi(x;b)$ Kim-forks over $M$. Then $\varphi(x;b)\models \bigvee_{j<n} \psi_j(x;c_j)$, where each $\psi_j(x;c_j)$ Kim-divides over $M$. 

By Theorem~\ref{thm:kimstrict}, let $q(y,z_1,\dots,z_n)$ be a global Kim-strict invariant type extending $\tp(bc_0\dots c_{n-1}/M)$, and let $I = (b^i,c_0^i, \dots, c_{n-1}^i)_{i< \omega}$ be a Morley sequence for $q$ over $M$. 

For all $j< n$, $I_j = (c_j^i)_{i< \omega}$ is also a Morley sequence over $M$ for a global Kim-strict invariant type, namely the restriction of $q$ to formulas in the single variable $z_j$. By New Kim's Lemma, $\psi_j(x;c_j)$ divides along $I_j$.

Suppose for contradiction that $\varphi(x;b)$ does not divide along $I_* = (b^i)_{i<\omega}$. Then there exists $a$ satisfying $\{\varphi(x;b^i):i<\omega\}$. For each $i<\omega$, since $b^ic_0^i\dots c_{n-1}^i\equiv_M bc_0\dots c_{n-1}$, there exists $j<n$ such that $\models \psi_j(a;c^i_j)$. By the pigeonhole principle, there is some $j< n$ such that for infinitely many $i<\omega$, $a$ satisfies $\psi_j(x;c^i_j)$. This contradicts the fact that $\psi_j(x;c_j)$ divides along $I_j$. Thus $\varphi(x;b)$ divides along $I_*$. Since $I_*$ is a Morley sequence over $M$ for the restriction of $q$ to formulas in the single variable $y$, $\varphi(x;b)$ Kim-divides over $M$. 
\end{proof}

\section{Examples}\label{sec:examples}

\subsection{Parametrized linear orders}\label{sec:dlop}

In this section, we introduce the theory $\DLO_p$ of parametrized dense linear orders without endpoints, and we show that it satisfies New Kim's Lemma. The choice of this example is motivated by the examples in Section~\ref{sec:diversity}: $\DLO_p$ is to $\DLO$ (Example~\ref{ex:dlo}) as $T^*_\feq$ (Example~\ref{ex:feq}) is to $T_E$ (Example~\ref{ex:eqrel}). 

The language $L$ has two sorts, $O$ and $P$, and one ternary relation $y <_x z$, where the subscript $x$ has type $P$ and $y$ and $z$ have type $O$. For an $L$-structure $A$, we write $A_P$ and $A_O$ for the two sorts. Let $L_{<}$ be the language $\{<\}$, where $<$ is a binary relation. Given $a\in A_P$, we write $A_a$ for the $L_<$-structure $(A_O,<_a)$. 

Let $\K$ be the class of all finite structures $A$ such that for all $a\in A_P$, $<_a$ is a linear order on $A_O$. This is a special case of the parametrization construction introduced in~\cite[Section 6.3]{ArtemNick}, applied to the class of finite linear orders. By~\cite[Lemma 6.3]{ArtemNick},  $\K$ is a Fra\"iss\'e class with disjoint amalgamation. Let $\DLO_p$ be the theory of its Fra\"iss\'e limit. By disjoint amalgamation, $\DLO_p$ has trivial $\acl$. By~\cite[Lemma 6.4]{ArtemNick}, if $M\models \DLO_p$, then for all $m\in M_P$, $M_m\models \DLO$. 

If $C\subseteq \monster_O$ and $\varphi(x)$ is an $L_{<}$-formula with parameters in $C$, then, for each $m \in \monster_{P}$, we write $\varphi_{m}(x)$ for the $L$-formula obtained by replacing each instance of $<$ with $<_{m}$.  Likewise, if $q(x)$ is a partial $L_{<}$-type over $C$, we write $q_{m}(x)$ for $\{\varphi_{m}(x) : \varphi(x) \in q\}$.  Note that $q_m(x)$ is a partial $L$-type over $Cm$.  

\begin{fact}[{\cite[Lemma 6.5]{ArtemNick}}]\label{fact:DLOp}
Suppose $C\subseteq \monster_O$, $(b_i)_{i\in I}$ is a family of distinct elements of $\monster_P$, and for each $i\in I$, $p^i(x)$ is a consistent non-algebraic $L_{<}$-type over $C$ in $\monster_{b_i}$. Then $\bigcup_{i\in I} p^i_{b_{i}}(x)$ is a consistent partial $L$-type over $C(b_i)_{i\in I}$.  
\end{fact}

Recall that a \emph{coheir sequence} over $A$ is a Morley sequence for a global type finitely satisfiable in $A$. The following lemma is a general fact that is easy and well-known.

\begin{lem}  \label{tuple lengthening}
Suppose $M \models T$ and $I = (a_{i})_{i < \omega}$ is a coheir sequence over $M$. Then given any $b$, there exists $(b_{i})_{i<\omega}$ such that $(a_{i},b_{i})_{i < \omega}$ is a coheir sequence over $M$ and $\tp(a_i b_i/M) = \tp(a_0 b/M)$ for all $i<\omega$.  
\end{lem}

\begin{proof}
Suppose $(a_i)_{i<\omega}$ is a coheir sequence for the global $M$-finitely satisfiable type $p(x)$. Let $N$ be an $|M|^+$-saturated model containing $M$. Let $a^*$ realize $p|N$, so $a^*\ind^u_M N$. By left extension for $\ind^u$, we can find $b^*$ such that $\tp(a^*b^*/M) = \tp(a_0b/M)$ and such that $a^*b^*\ind^u_M N$. By saturation of $N$, $\tp(a^*b^*/N)$ has a unique global $M$-invariant extension $q(x,y)$, which is finitely satisfiable in $M$. Likewise, $p(x)\subseteq q(x,y)$, since the restriction of $q$ to formulas in context $x$ is the unique global $M$-invariant extension of $\tp(a^*/N) = p|N$. Let $(a^*_ib^*_i)_{i<\omega}$ be a Morley sequence for $q$ over $M$. Since $(a^*_i)_{i<\omega}$ and $(a_i)_{i<\omega}$ are both Morley sequences for $p$ over $M$, there is an automorphism $\sigma$ of $\monster$ over $M$ such that $\sigma(a^*_i) = a_i$ for all $i$. Let $b_i = \sigma(b^*_i)$. 
\end{proof}

\begin{lem}\label{lem:DLOp}
Let $M\models \DLO_p$. Then $A\ind^{Kd}_M B$ if and only if: 
\begin{enumerate}
\item $A\cap B\subseteq M$, and 
\item for every $m\in M_P$, and for all $b<_ma <_m b'$ with $a\in A_O\setminus M_O$ and $b,b'\in B_O\setminus M_O$, there exists $m'\in M_O$ such that $b<_m m'<_m b'$ (i.e., $b$ and $b'$ live in different $<_m$-cuts in $M_O$).
\end{enumerate}
\end{lem}

Condition (2) in the statement of Lemma~\ref{lem:DLOp} can be more succinctly stated as: for every $m\in M_P$, $A_O\ind^f_{M_O} B_O$ in the $L_{<}$-structure $\monster_m$. Nevertheless, we will prove and use the more concrete characterization. 

\begin{proof}
Suppose $A\ind^{Kd}_M B$. In any theory, $A\ind^{Kd}_M B$ implies $A\cap B\subseteq M$, so we have (1). For (2), assume for contradiction that $b<_m a <_m b'$, with $m\in M_P$, $a\in A_O\setminus M_O$, and $b,b'\in B_O\setminus M_O$, and $b$ and $b'$ live in the same $<_m$-cut in $M_O$, i.e., there is no $m'\in M_O$ such that $b<_m m' <_m b'$. 

We will find a global type $q(y,y')$ extending $\tp(bb'/M)$ and finitely satisfiable in $M$, such that the formula $\varphi(x;b,b')\colon b <_m x <_m b'$ divides along Morley sequences for $q$ over $M$. We may assume that the set $C = \{c\in M_O\mid c <_m b\}$ is non-empty and has no greatest element. The other case, when the set $D = \{d\in M_O\mid  b<_m d\}$ is non-empty  and has no least element, is symmetrical. 

Consider the filter on $M_O^{yy'}$ generated by \[\{\psi(M): \psi(y,y')\in \tp(bb'/M)\}\cup \{(e,e')\mid e'\in C\}.\] By quantifier elimination, a set $Y$ in this filter contains the intersection of:
\begin{enumerate}
\item $\{(e,e')\mid e'\in C\}$,
\item A set $\{(e,e')\mid c <_m e <_m e' <_m d\}$ for some $c\in C$ and $d\notin C$, or $\{(e,e')\mid c <_m e <_m e'\}$ for some $c\in C$, and
\item finitely many non-empty sets in $M_O^{yy'}$, each defined in terms of an order $<_{m'}$ for $m'\neq m$ in $M_P$. 
\end{enumerate}
Since $C$ has no greatest element, we can pick some $c'\in C$ with $c<_m c'$. Then replacing (1) and (2) in the intersection with $\{(e,e')\mid c <_m e <_m e' <_m c'\}$, the intersection of these sets is non-empty, by the extension axioms for the Fra\"iss\'e limit, and contained in $Y$. Thus the filter is proper and extends to an ultrafilter $\D$. 

Let $q = \Av(\D,\monster)$. Suppose $I = (b_i,b'_i)_{i< \omega}$ is a Morley sequence for $q$ over $M$. Since each $b_i$ realizes $\tp(b/M)$, \[\{(e,e')\in M_O^{yy'} \mid e' <_m b_i\} = \{(e,e')\in M_O^{yy'}\mid e'\in C\}\in \D.\]
So $b_{i+1} <_m b_{i+1}' <_m b_i$ for all $i< \omega$. Thus the $<_m$-intervals $(b_i,b'_i)$ are pairwise disjoint, and $b<_m x <_m b'$ divides along $I$. 

Let $b^*$ be a tuple enumerating $B\setminus \{b,b'\}$. By Lemma~\ref{tuple lengthening}, there exists $(b^*_i)_{i< \omega}$ such that $J = (b_i,b'_i,b^*_i)_{i< \omega}$ is a coheir sequence over $M$ and $\tp(b_i,b'_i,b^*_i) = \tp(b,b',b_*) = \tp(B/M)$ for all $i<\omega$. The formula $b<_m x <_m b'$ is contained in $\tp(A/MB)$ and divides along $J$, which contradicts $A\ind^{Kd}_M B$. 

\medskip

Conversely, suppose conditions (1) and (2) hold. We may assume that $A$ is disjoint from $M$ (and hence also from $B$, by (1)), since $(A\setminus M)\ind^{Kd}_M B$ implies $A\ind^{Kd}_M B$. 

Let $p(x,x',y) = \tp(AB/M)$, where $x$ enumerates $A_O$, $x'$ enumerates $A_P$, and $y$ enumerates $B$. Let $(B_i)_{i< \omega}$ be a Morley sequence for a global $M$-invariant type extending $\tp(B/M)$. Let $C = M\cup \bigcup_{i< \omega} B_i$. It suffices to show that $q(x,x') = \bigcup_{i< \omega} p(x,x',B_i)$ is a consistent partial type over $C$. 
 
Let $q_O(x)$ be the subset of $q$ that only mentions the variables $x$ (those of type $O$). For each $c\in C_P$, let $q_c(x)$ be set of atomic and negated atomic formulas in $q_O(x)$ involving the relation $<_c$. Then there is a partial $L_{<}$-type $q^c(x)$ over $C_O$ in $\monster_c$ such that $(q^c)_c$ is equivalent to $q_c$. We will show that each $q^c(x)$ is consistent. 

If $c\notin M$, then since the $B_i$ are pairwise disjoint over $M$, there is a unique $i< \omega$ such that $c\in (B_i)_P$. Then $q_c(x)$ is contained in $p(x,x',B_i)$, which is consistent, and hence $q^c(x)$ is consistent as well.

Suppose $c\in M$, and assume for contradiction that $q^c(x)$ is inconsistent. By compactness and density of $\monster_c$, there is some variable $z$ from $x$ and some $b_i\in B_i$ and $b'_j\in B_j$ for $i,j< \omega$ such that $b_i \leq b'_j$ in $\monster_c$, but $q^c(x)$ entails $b'_j < z < b_i$. Let $b$ and $b'$ be the elements of $B$ corresponding to $b_i$ and $b'_j$, respectively, and let $a$ be the element of $A$ corresponding to the variable $z$. Then $b' < a < b$, so by (2) there is some $m'\in M_O$ such that $b' < m' < b$ in $\monster_c$. But since $B_i\equiv_M B_j\equiv_M B$, $b_j'< m' < b_i$, contradicting $b_i \leq b'_j$. 

Since $A$ is disjoint from $B$, each type $q^c(x)$ is non-algebraic, so by Fact~\ref{fact:DLOp}, $\bigcup_{c\in C_P} q_c(x)$ is consistent. By quantifier elimination, $q_O(x)$ is consistent. 

Let $A'$ realize $q_O(x)$. It remains to show that $q(A',x')$ is consistent. Each variable in $x'$ is of type $P$. Since each atomic formula contains at most one variable of type $P$, it suffices to show that for each variable $z$ in $x'$, the set $r(z)$ of all atomic and negated atomic formulas from $q(A',x')$ involving the relation $<_z$ is consistent. 

The type $r(z)$ specifies a linear order on $A'M_O$, which extends to a linear order on $A'M_O(B_i)_O$ for all $i< \omega$. Using the amalgamation property for linear orders, we can find a linear order on $A'C_O$ extending each of the given linear orders. By compactness and the extension axioms for the Fra\"iss\'e limit, we can find $c\in \monster_P$ such that $<_c$ induces this linear order on $A'C_O$. This completes the proof. 
\end{proof}

\begin{thm}
$\DLO_p$ satisfies New Kim's Lemma. 
\end{thm}
\begin{proof}
Suppose $\varphi(x;b)$ Kim-divides over $M\models T$. To show that $\varphi(x;b)$ universally Kim-strictly divides, let $p(y)$ be a global Kim-strict $M$-invariant type extending $\tp(b/M)$, and let $I = (b_n)_{n< \omega}$ be a Morley sequence for $p$. Suppose for contradiction that $\{\varphi(x;b_n):n<\omega\}$ is consistent, realized by $a$. By Ramsey's theorem, compactness, and an automorphism, we may assume that $(b_n)_{n<\omega}$ is $Ma$-indiscernible. Now it suffices to show that $a\ind^{Kd}_M b_0$, since this will contradict the fact that the Kim-dividing formula $\varphi(x;b_0)$ is in $\tp(a/Mb_0)$. 

Let $A$ be the set enumerated by $a$, and let $B_n$ be the set enumerated by $b_n$ for all $n$. Since $B_1\ind^i_M B_0$, $B_1\cap B_0\subseteq M$. If $c\in A\cap B_0$, then since $AB_0\equiv_M AB_1$, also $c\in B_1$, so $c\in M$. Thus $A\cap B_0\subseteq M$.  

Now suppose $m\in M_P$ and $d_0<_m c<_m d_0'$, with $c\in A_O\setminus M_O$ and $d_0,d_0'\in (B_0)_O\setminus M_O$. Suppose for contradiction that there is no $m'\in M_O$ such that $d_0<_m m' <_m d_0'$. Let $d_1$ and $d_1'$ be the elements of $(B_1)_O$ corresponding to $d_0$ and $d_0'$ in $(B_0)_O$. Since $B_1\equiv_{MA} B_0$, $d_1<_m c <_m d_1'$, and there is no $m'\in M_O$ such that $d_1<_m m' <_m d_1'$. 

Since $p$ is Kim-strict, $B_0\ind^{Kd}_M B_1$ and $B_1\ind^{Kd}_M B_0$. By Lemma~\ref{lem:DLOp}, $d_0$, $d_0'$, $d_1$, and $d_1'$ are distinct,  neither $d_0$ nor $d_0'$ are in the $<_m$-interval $(d_1,d_1')$, and neither $d_1$ nor $d_1'$ are in the $<_m$-interval $(d_0,d_0')$. It follows that the $<_m$-intervals $(d_0,d_0')$ and $(d_1,d_1')$ are disjoint. This contradicts the fact that $c$ is in both of them. 

So there is $m'\in M_O$ such that $d_0<_m m' <_m d_0'$. By Lemma~\ref{lem:DLOp}, $A\ind^{Kd}_M B_0$. 
\end{proof}

\subsection{Bilinear forms over real closed fields}

Let $T^{\RCF}_{\infty}$ be the two-sorted theory of an infinite-dimensional vector space over a real closed field with a bilinear form, which is assumed to be either alternating and non-degenerate, or symmetric and positive-definite. This is really two theories, one for each type of bilinear form, but our arguments are identical in both cases so we will not notationally distinguish them.  The language has a sort $V$ for the vector space, equipped with the language of abelian groups, a sort $R$ for the real closed field of scalars, equipped with the language of ordered rings, a function symbol $\cdot\colon R\times V\to V$ for scalar multiplication, and a function symbol $[-,-]\colon V\times V\to R$ for the bilinear form. 

By \cite{Granger}, $T^{\RCF}_{\infty}$ is the model companion of the theory of a vector space over a real closed field with an alternating (or symmetric and positive-definite) bilinear form. By \cite{JAN}, this theory additionally has quantifier-elimination in an expanded language, containing, for each $n$, a predicate $I_n$ on $V^{n}$, such that $I_n(v_1,\dots,v_n)$ holds if and only if $v_{1}, \ldots, v_{n}$ are linearly independent, as well as $(n+1)$-ary ``coordinate functions'' $F_{n,i} : V^{n+1} \to R$ for each $1 \leq i \leq n$.  These functions are which are interpreted so that, if $v_{1}, \ldots, v_{n}$ are linearly independent and $w = \sum_{i = 1}^{n} \alpha_{i} v_{i}$, then $F_{n,i}(\overline{v},w) = \alpha_{i}$, and $F_{n,i}(\overline{v},w) = 0$ otherwise.

When $A$ is a subset of $\monster\models T^\RCF_\infty$, we write $A_R$ for the elements of the field sort and $A_V$ for the elements of the vector space sort. 

\begin{rem}\label{rem:stablyembedded}
As a consequence of quantifier elimination and elementary linear algebra, the field sort $R$ is stably embedded. More precisely, suppose $C$ is a substructure of $\monster$. If $\varphi(x)$ is a formula with parameters from $C$ such that every variable is in the field sort $R$, then $\varphi(x)$ is equivalent to a formula $\psi(x)$ in the language of ordered rings with parameters from $C_R$. Consequently, for any tuple $a$ from $\monster_R$ and any substructure $C$, $\tp_{\RCF}(a/C_R)$ entails $\tp(a/C)$. 
\end{rem}

If $W$ is a set of vectors, we write $\langle W\rangle$ for the linear span of $W$ with scalars from the field $\monster_R$ (so $\langle W\rangle$ is a large set). By $\dim(W)$, we mean the dimension of $\langle W\rangle$ as a vector space over $\monster_R$. 

Suppose $A$, $B$, and $C$ are substructures of $\monster$. We write $A \ind^{\RCF}_{C} B $ to mean that $A_R$ and $B_R$ are forking-independent over $C_R$ in the reduct of $\monster_R$ to a model of $\RCF$. We write $A\ind^V_C B$ to mean $\langle A_V\rangle \cap \langle B_V\rangle \subseteq \langle C_V\rangle$. Our goal is to show that $T_\infty^\RCF$ satisfies New Kim's Lemma, which will involve characterizing $\ind^{Kd}$ in this theory in terms of $\ind^\RCF$ and $\ind^V$. The argument is the analogue of \cite[Proposition 9.37]{kaplan2017kim} (incorporating the corrections of \cite[Proposition 8.12]{JAN}).  A similar characterization of Kim-independence in the theory of a bilinear form on a vector space over an NSOP$_{1}$ field occurs in \cite{bossut2023note}.  

We begin with another general lemma which, in conjunction with Lemma~\ref{tuple lengthening}, will allow us to upgrade a coheir sequence in $\tp_\RCF(B_R/M_R)$ in $\RCF$ to a coheir sequence in $\tp(B/M)$ in $T^\RCF_\infty$. 

\begin{lem}\label{coheir expansion}
Suppose $L \subseteq L'$ are languages, $T'$ is an $L'$-theory and $T = T' \upharpoonright L$.  If $A \subseteq B$ and $I = (c_{i})_{i < \omega}$ is a coheir sequence over $A$ in $T$, then there is $I' \models \text{tp}_{L}(I/A)$ which is a coheir sequence in $T'$ over $B$.  
\end{lem}

\begin{proof}
If $(c_{i})_{i < \omega}$ is a coheir sequence over $A$ in $T$, there is some ultrafilter $\D$ on $A^{n}$, where $n$ is the length of $c_{0}$, such that $I$ is a Morley sequence over $A$ in the global $A$-finitely satisfiable type $\Av_{L}(\D, \monster)$. To see this, stretch $I$ to $(c_{i})_{i < \omega + 1}$ and observe that the family of sets $\{\varphi(A;c_{<\omega}) : \varphi(x;c_{<\omega}) \in \mathrm{tp}(c_{\omega}/Ac_{<\omega})\} \subseteq \Power(A^{n})$ generates a filter and hence extends to an ultrafilter $\D$.  It is easily checked that this $\D$ works.  Let $\E$ be the ultrafilter on $B^{n}$ induced by $\D$, i.e. a subset $X \subseteq B^{n}$ satisfies $X \in \E$ if and only if $X \cap A^{n} \in \D$.  Then we can take $I'$ to be Morley over $B$ in the global $B$-finitely satisfiable type $\Av_{L'}(\E, \monster)$. 
\end{proof}

\begin{lem}\label{lem:kimrcf}
If $M \models T^{\RCF}_{\infty}$ and $A\ind^{Kd}_M B$, then $A\ind^\RCF_M B$. 
\end{lem}
\begin{proof}
Because $\RCF$ is an $\NTP_{2}$ theory, any dividing formula divides along some coheir sequence by~\cite[Lemma 3.12]{chernikov2012forking}. So if $A \nind^{\RCF}_{M} B$, then there is a formula $\varphi(x;b)$ in $\tp_{\RCF}(A_R/M_RB_R)$ and a coheir sequence $I = (B_{i})$ over $M_{R}$ in $\tp_{\RCF}(B_R/M_R)$ such that $\varphi(x;b)$ divides along $I$. By Lemmas~\ref{coheir expansion} and~\ref{tuple lengthening}, there is a coheir sequence $I' = (B'_{i})_{i <  \omega}$ over $M$ in $\tp(B/M)$ such that $((B'_{i})_{R})_{i < \omega} \equiv^\RCF_{M_{R}} I$. Then $\varphi(x;b)$ divides along $I'$, and $A \nind^{Kd}_{M} B$. 
\end{proof}

\begin{lem}\label{lem:kimvs}
Suppose $M\models T^\RCF_\infty$. 
\begin{enumerate}
\item If $A\ind^u_M B$, then $A\ind^V_M B$. 
\item If $(B_i)_{i<\omega}$ is a $\ind^V_M$-independent sequence (i.e.,  $B_i\ind^V_M B_0\dots B_{i-1}$ for all $i<\omega$), and there exists $A'$ such that $A'B_i\equiv_M AB$ for all $i<\omega$, then $A\ind^V_M B$.
\item If $A\ind^{Kd}_M B$, then $A\ind^V_M B$.  
\end{enumerate}
\end{lem}
\begin{proof}
Suppose that $A\nind^V_M B$. Then $\langle A_V\rangle \cap \langle B_V\rangle \not\subseteq \langle M_V\rangle$, so there exists a vector $v$, a finite linearly independent tuple $a$ from $A_V$, and a finite linearly independent tuple $b$ from $B_V$ such that $v\in \langle a\rangle \cap \langle b\rangle$ and $v\notin \langle M_V\rangle$. Let $C = \langle b\rangle \cap \langle M_V\rangle$, and note that $C$ is a subspace of the finite-dimensional space $\langle b\rangle$. Let $c$ be a finite basis for $C$. Note that the formula $\varphi(x;b,c)$: \[\exists w\, (I_{|a|}(x)\land  \lnot I_{|a|+1}(x,w)\land \lnot I_{|b|+1}(b,w)\land I_{|c|+1}(c,w)),\] which asserts that $x$ is linearly independent and $\langle x\rangle \cap \langle b\rangle \not\subseteq \langle c\rangle$, is in $\tp(a/Mb)$. With the above notation set, we now prove (1) and (2).  

For (1), assume for contradiction that $A\ind^u_M B$. Since $\tp(a/Mb)$ is finitely satisfiable in $M$, there is some $a'\in M_V$ satisfying $\varphi(a';b,c)$. Let $w'$ be the witness to the existential quantifier. Then $w'\in \langle a'\rangle\subseteq \langle M_V\rangle$ and $w'\in \langle b\rangle$, so $w'\in \langle b\rangle\cap \langle M_V\rangle = C$. But $w'\notin \langle c\rangle$, contradiction.

For (2), assume for contradiction that there exists a $\ind^V_M$-independent sequence $(B_i)_{i<\omega}$ and $A'$ such that $A'B_i\equiv_M AB$ for all $i<\omega$. Let $(b_i)_{i<\omega}$ be the restriction of this sequence to the tuples $b_i$ from $B_i$ corresponding to the tuple $b$ in $B$, and let $a'$ be the tuple from $A'$ corresponding to the tuple $a$ in $A$. Let $k = \dim(\langle a'\rangle) = |a'|$, and let $v_0,\dots,v_k$ be such that $v_i\in \langle a'\rangle \cap \langle b_i\rangle \setminus \langle c\rangle$ for all $i <k+1$. Since these $k+1$ vectors are all in $\langle a'\rangle$, they are not linearly independent, and we can write one of them, say $v_j$, as a linear combination of $v_0,\dots,v_{j-1}$. Then $v_j\in \langle b_j\rangle\cap \langle b_0,\dots,b_{j-1}\rangle\setminus \langle c\rangle$. But since $b_j\ind^V_M b_0\dots b_{j-1}$, $\langle b_j\rangle\cap \langle b_0,\dots,b_{j-1}\rangle \subseteq \langle b_j\rangle\cap \langle M_V\rangle = \langle c\rangle$, contradiction. 

For (3), let $(B_i)_{i<\omega}$ be a coheir sequence in $\tp(B/M)$. Since $A\ind^{Kd}_M B$, by compactness there exists $A'$ such that $A'B_i\equiv_M AB$ for all $i<\omega$. By (1), $(B_i)_{i<\omega}$ is a $\ind^V_M$-independent sequence, and by (2), $A\ind^V_M B$. 
\end{proof}

\begin{thm} \label{Kimchar} 
If $M \models T^{\RCF}_{\infty}$, $A = \mathrm{acl}(AM)$, $B = \mathrm{acl}(BM)$, then $A \ind^{Kd}_{M} B$ if and only if $A \ind^{\RCF}_{M} B$ and $A\ind^V_M B$. 
\end{thm}

\begin{proof}
One direction is Lemma~\ref{lem:kimrcf} and Lemma~\ref{lem:kimvs}(3).

In the other direction, suppose that $A \ind^{\RCF}_{M} B$ and $A\ind^V_M B$. Let $(B_{i})_{i < \omega}$ be a Morley sequence over $M$ for a global $M$-invariant type extending $\tp(B/M)$. Since $A\ind^\RCF_M B$, we can find $A'_R$ such that $A'_R(B_i)_R \equiv^\RCF_{M_R} A_RB_R$ for all $i<\omega$. By Remark~\ref{rem:stablyembedded}, $A'_RB_i \equiv_{M} A_RB$ for all $i<\omega$. Let $\tilde{R}$ be the field $(\acl(A'_R(B_i)_{i<\omega}))_R$.  

Let $\overline{m} = (m_{i})_{i < \alpha}$ be a tuple from $M_V$ which is a basis of $\langle M_V \rangle$.  Choose $\overline{a} = (a_{i})_{i < \beta}$ from $A_V$ such that $\overline{a}\overline{m}$ is a basis of $\langle A_V \rangle$ and choose $\overline{b}_{i} = (b_{i,j})_{j < \gamma}$ from $(B_i)_V$ such that $\overline{m} \overline{b}_{i}$ is a basis of $\langle (B_{i})_V \rangle$. Since $(B_i)_{i<\omega}$ is a $\ind^i_M$-independent sequence, by Lemma~\ref{lem:kimvs}(3) it is also a $\ind^V_M$-independent sequence. Thus $\overline{m}$ and $(\overline{b}_{i})_{i < \omega}$ are linearly independent. Let $\tilde{V} = \langle \overline{m}(\overline{b}_{i})_{i < \omega} \rangle_{\tilde{R}}$, the vector space over $\tilde{R}$ spanned by this basis. Note that, unlike $\langle \overline{m}(\overline{b}_{i})_{i < \omega} \rangle$, this is a small set, and it contains $(B_i)_{i<\omega}$, since $\tilde{R}$ contains the values of the coordinate functions $F_{n,i}$ on tuples from $(B_i)_{i<\omega}$. Let $\tilde{N}$ be the substructure of $\monster$ with $\tilde{N}_R = \tilde{R}$ and $\tilde{N}_V = \tilde{V}$. Note that if we give the symbols $\theta_n$ and $F_{n,i}$ their intended interpretations in $\tilde{N}$, they agree with the interpretations of these symbols in $\monster$. 

Let $\overline{a}' = (a'_i)_{i<\beta}$ be a tuple of new vectors (not in $\monster_V$) of the same length as $\overline{a}$. Let $W$ be the $\tilde{R}$-vector space extending $\tilde{V}$ with basis $\overline{a}'$, $\overline{m}$, and $(\overline{b}_{i})_{i < \omega}$. We build a structure $N$ extending $\tilde{N}$ with $N_R = \tilde{N}_R = \tilde{R}$ and $N_V = W$. The field structure and vector space structure have been determined, so it remains to define the bilinear form $[-,-]^N$. To do this, it suffices to define the form on every pair of basis vectors for $W$ such that at least one comes from $\overline{a}'$, and extend linearly. 

For all $i < \alpha$, $i'<\beta$, ,$j'<\beta$, $j < \omega$, and $k < \gamma$, set
\begin{align*}
[a'_{i'},a'_{j'}]^N &= [a_{i'},a_{j'}]^\monster\\
[a'_{i'},m_i]^N &= [a_{i'},m_i]^\monster\\
[a'_{i'}, b_{j,k}]^{N} & = [a_{i}, b_{0,k}]^{\monster}.
\end{align*}
These conditions uniquely determine a bilinear form on all pairs of vectors from $W$, which is alternating or symmetric and positive-definite, as required by $T^\RCF_\infty$. We can extend the language to include the $\theta_n$ and $F_{n,i}$ in the natural way, and the interpretations of these symbols agree with those on $\tilde{N}$, since $\tilde{N}_R = N_R$. 

Now we can embed $N$ into $\monster$ over $\tilde{N}$.  Let $A'_V$ be the image under this embedding of the subset of $N$ corresponding to $A_V$, and let $A' = (A'_R,A'_V)$.  It follows by construction and quantifier elimination that $A'B_{i} \equiv_{M} AB$ for all $i<\omega$.  Thus $A\ind^{Kd}_M B$. 
\end{proof}

\begin{thm}
The theory $T^{\RCF}_{\infty}$ satisfies New Kim's Lemma.
\end{thm}

\begin{proof}
Let $M \models T^{\RCF}_{\infty}$ and suppose $\varphi(x;b)$ Kim-divides over $M$. Let $I = (b_{i})_{i < \omega}$ be a Morley sequence over $M$ for a global Kim-strict $M$-invariant type $q(y)\supseteq\tp(b/M)$. We would like to show that $\varphi(x;b_i)$ divides along $I$. Assume, towards contradiction, that there exists $a$ realizing $\{\varphi(x;b_{i}) : i < \omega\}$. By Ramsey's theorem, compactness, and an automorphism, we may assume that $(b_{i})_{i < \omega}$ is indiscernible over $A = \acl(Ma)$. For each $i < \omega$, let $B_{i} = \mathrm{acl}(Mb_{i})$, with each $B_i$ enumerated in such a way that $(B_i)_{i<\omega}$ remains indiscernible over $A$. 

Since $(b_i)_{i<\omega}$ is a $\ind^i_M$-independent sequence, it is a $\ind^K_M$-independent sequence, and thus $(B_i)_{i<\omega}$ is a $\ind^K_M$-independent sequence. By Lemma~\ref{lem:kimvs}(3), $(B_i)_{i<\omega}$ is a $\ind^V_M$-independent sequence, and since $AB_i\equiv_M AB_0$ for all $i<\omega$, $A\ind^V_M B_0$ by Lemma~\ref{lem:kimvs}(2). 

We now claim that $((B_{i})_R)_{i < \omega}$ is a (Kim-)strict Morley sequence over $M_R$ in $\RCF$. Let $N$ be an $|M|^+$-saturated model containing $M$ and $(B_i)_{i<\omega}$. Let $b_\omega$ realize $q|_{N}$, and let $B_\omega = \acl(Mb_\omega)$. Since $(b_i)_{i\leq \omega}$ is a Morley sequence over $M$, and hence $M$-indiscernible, and $(B_i)_{i < \omega}$ is $M$-indiscernible, we can enumerate $B_\omega$ in such a way that $(B_i)_{i\leq \omega}$ remains $M$-indiscernible.

Since $q$ is Kim-strict, $b_\omega\ind^K_M N$ and $N\ind^K_M b_\omega$, so $B_\omega \ind^K_M N$ and $N \ind^K_M B_\omega$, and hence $B_\omega \ind^{\RCF}_{M} N$ and $N \ind^{\RCF}_{M} B_\omega$, by Theorem~\ref{Kimchar}. Since $\RCF$ is an $\NIP$ theory, $\ind^{i}_M = \ind^{f}_M$ in $\RCF$ (see~\cite[Corollary 5.22]{simon2015guide}). Thus $\mathrm{tp}_{\RCF}((B_\omega)_R/N_R)$ extends to a global $M_R$-invariant type $q_*$ which is strict over $M_R$ in $\RCF$. Indeed, suppose for contradiction that $C_R\subseteq \monster_R$, $B'_R\models q_*|_{N_RC_R}$, and $C_R\nind^f_{M_R} B'_R$ in $\RCF$. Then $c\nind^f_{M_R} B'_R$ for some finite tuple $c$ from $C_R$, whose type over $M_R$ is realized by $c'\in N_R$. Then $c'\nind^f_{M_R} B'_R$ in $\RCF$ by invariance of $q_*$, contradicting $N \ind^{\RCF}_{M} B_\omega$.

By $M$-indiscernibility of $(B_i)_{i\leq \omega}$, $(B_{i})_R \models q_{*}|_{M_R(B_{<i})_R}$ for all $i$, so $((B_{i})_R)_{i < \omega}$ is a strict Morley sequence over $M_R$ in $\RCF$. Since $((B_{i})_R)_{i < \omega}$ is $A_R$-indiscernible, it follows that $A \ind^{\RCF}_{M} B_{0}$ by the NTP$_{2}$ Kim's Lemma (Theorem~\ref{thm:ntp2kim}). 

Since $A\ind^\RCF_M B_0$ and $A\ind^V_M B_0$, by Theorem~\ref{Kimchar}, $A\ind^{Kd}_M B_0$. This contradicts the fact that $\tp(A/MB_0)$ contains the formula $\varphi(x;b_0)$, which Kim-divides over $M$, since $b_0\equiv_M b$.   
\end{proof}

\subsection{Non-example: the Henson graph}

The Henson graph, or generic triangle-free graph, is the Fra\"iss\'e limit of the class of finite triangle-free graphs. Its complete theory $T_\triangle$ is $\SOP_3$ and $\NSOP_4$. Conant analyzed forking and dividing in $T_\triangle$ in detail in~\cite{conant2014forking}. We will use the following characterization of $\ind^f$.  

\begin{fact}[{\cite[Theorem 5.3]{conant2014forking}}]\label{fact:hensonforking}
Suppose that $A$ and $B$ are sets in $\monster\models T_\triangle$ and $M\models T_\triangle$. Then $A\ind^f_M B$ if and only if $A\cap B\subseteq M$ and for all $a\in A$ and $b\neq c\in B\setminus M$, if $aRb$ and $aRc$, then there exists $m\in M$ such that $mRb$ and $mRc$.
\end{fact}

We will show that a very weak variant of Kim's Lemma fails in $T_\triangle$: strict dividing does not imply universal strict dividing. Since strict dividing implies Kim-dividing and universal Kim-strict dividing implies universal strict dividing, it follows that $T_\triangle$ fails to satisfy New Kim's Lemma. 

\begin{thm}
Modulo $T_\triangle$, there is a formula which strictly divides but does not universally strictly divide. Thus $T_\triangle$ does not satisfy New Kim's Lemma. 
\end{thm}
\begin{proof}
Let $M\models T$. Let $b$ and $c$ be elements of $\monster \setminus M$ with $\lnot b R c$, such that\footnote{Really, all we will use is that the set of neighbors of $b$ in $M$ is non-empty and disjoint from the set of neighbors of $c$ in $M$.} $b$ has a single neighbor in $M$, call it $m$,  and $c$ has no neighbors in $M$. Consider the formula $\varphi(x;b,c)\colon xRb\land xRc$. It suffices to find two strict global $M$-invariant types $p(y,z)$ and $q(y,z)$ extending $\tp(b,c/M)$ such that $\varphi(x;b,c)$ divides along Morley sequences for $p$ but does not divide along Morley sequences for $q$. 

Let $p(y,z)$ extend $\tp(bc/M)$ by including, for each $d\in \monster\setminus M$, $y\neq d$, $z\neq d$ and $\lnot y Rd$. Additionally,  include $zRd$ if $d\models \tp(b/M)$ and $\lnot zRd$ otherwise. We claim this defines a consistent partial type. Any inconsistency would come from a triangle involving the variables and elements of $\monster$. Such a triangle cannot contain $y$, since $\lnot yRz$ and $y$ has an edge to exactly one element of $\monster$, namely $m$. Since $z$ only has edges to realizations of $\tp(b/M)$, any triangle containing $z$ contains two realizations of $\tp(b/M)$. But no two realizations of $\tp(b/M)$ are adjacent, since they are both adjacent to $m$. 

By quantifier elimination, this partial type determines a complete $M$-invariant type over $\monster$. Letting $I = (b_i,c_i)_{i< \omega}$ be a Morley sequence for $p$ over $M$, $\varphi(x;b,c)$ divides along $I$, since $\{\varphi(x;b_1,c_1),\varphi(x;b_2,c_2)\}$ entails $\{xRb_1,xRc_2\}$, and $b_1Rc_2$. 

Now let $q(y,z)$ extend $\tp(bc/M)$ by including, for each $d\in \monster\setminus M$, $y\neq d$, $z\neq d$, $\lnot yRd$, and $\lnot zRd$. This defines a consistent partial type, since the only edge from a variable to an element of $\monster$ is the single edge from $y$ to $m$. Again, by quantifier elimination, this determines a complete $M$-invariant type over $\monster$. And if $J = (b_i,c_i)_{i< \omega}$ is a Morley sequence for $q$ over $M$, then $\varphi(x;b,c)$ does not divide along $J$. Indeed, since there are no edges among the vertices $\{b_i,c_i:i<\omega\}$, $\{\varphi(x;b_i,c_i):i<\omega\}$ does not induce any triangles.  

It remains to show that both $p$ and $q$ are strict. Let $A\subseteq \monster$, $b_0,c_0\models p|_{MA}$, and $b_1,c_1\models q|_{MA}$. We would like to show that for $i\in \{0,1\}$, $A\ind^f_M b_ic_i$. In each case, $A\cap \{b_i,c_i\} = \emptyset\subseteq M$, and there is no $a\in A$ such that $aRb_i$ and $aRc_i$ (since $b_i$ is not adjacent to any element of $A\setminus M$, and $c_i$ is not adjacent to any element of $M$). By Fact~\ref{fact:hensonforking}, $A\ind^f_M b_ic_i$.
\end{proof}

\section{Syntax}\label{sec:syntax}

In this section, we isolate a tree property, provisionally called $\BTP$, which generalizes $\TP_2$ and $\SOP_1$, and we show that $\NBTP$ theories satisfy New Kim's Lemma. We also show that $\NBTP$ theories are $\NATP$. We have not succeeded in proving that New Kim's Lemma characterizes $\NBTP$ theories. 

For ordinals $\alpha,\beta\leq \omega$, write $\alpha^{<\beta}_*$ for the forest obtained by removing the root from $\alpha^{<\beta}$. 

\begin{itemize}
\item A \emph{left-leaning path} in $\alpha^{<\beta}_*$ is a sequence $(\lambda_n)$ such that if $\lambda_n = \eta^\frown \langle i\rangle$, then $\eta^\frown \langle j\rangle \vartriangleleft \lambda_{n+1}$ for some $j \leq i$. 
\item A \emph{right-veering path} in $\alpha^{<\beta}_*$ is a sequence $(\rho_n)$ such that if $\rho_n = \eta^\frown \langle i\rangle$, then $\eta^\frown \langle j\rangle  \trianglelefteq \rho_{n+1}$ for some $j > i$. 
\end{itemize}

Note that to get to the next element in a left-leaning path, one \emph{optionally} moves leftward to a sibling and then moves \emph{strictly} upward to a descendent, while in a right-veering path, one moves \emph{strictly} rightward to a sibling, and then \emph{optionally} moves upward to a descendent. 

\begin{defn}
A formula $\varphi(x;y)$ has $k$-$\BTP$ (\emph{$k$-bizarre tree property}) with $k< \omega$ if there exists a forest of tuples $(a_\eta)_{\eta\in \omega^{<\omega}_*}$ satisfying the following conditions: 
\begin{itemize}
\item For every left-leaning path $(\lambda_n)_{n< \omega}$, $\{\varphi(x;a_{\lambda_n}): n< \omega\}$ is consistent. 
\item For every right-veering path $(\rho_n)_{n\in\omega}$, $\{\varphi(x;a_{\rho_n}): n< \omega\}$ is $k$-inconsistent. 
\end{itemize}

A theory $T$ has $\BTP$ if there is some formula $\varphi(x;y)$ and some $k< \omega$ such that $\varphi$ has $k$-$\BTP$. Otherwise, $T$ is $\NBTP$. 
\end{defn}

\begin{thm}
Suppose $T$ is $\NBTP$. Then $T$ satisfies New Kim's Lemma.
\end{thm}
\begin{proof}
We prove the contrapositive. If New Kim's Lemma fails, then we have a formula $\varphi(x;b)$, a model $M\models T$, and global $M$-invariant types $p(y)$ and $q(y)$ extending $\tp(b/M)$ such that $p(y)$ is Kim-strict and $\varphi(x;b)$ divides along Morley sequences for $q$ but not along Morley sequences for $p$. Fix $k< \omega$ such that if $(b_i)_{i< \omega}$ is a Morley sequence for $q$, then $\{\varphi(x;b_i):i< \omega\}$ is $k$-inconsistent. 

For arbitrary $m$ and $n$ in $\omega$, we will build a finite forest $(a_\eta)_{\eta\in m^{< n}_*}$ such that:
\begin{itemize}
\item For every left-leaning path $(\lambda_i)_{1\leq i\leq \ell}$ in $m^{< n}_*$, $(a_{\lambda_\ell},\dots,a_{\lambda_1})$ starts a Morley sequence for $p$ over $M$, and hence $\{\varphi(x;a_{\lambda_i}):1\leq i\leq \ell\}$ is consistent. 
\item For every right-veering path $(\rho_i)_{1\leq i \leq \ell}$ in $m^{<n}_*$,  $(a_{\rho_\ell},\dots,a_{\rho_1})$ starts a Morley sequence for $q$ over $M$, and hence $\{\varphi(x;a_{\lambda_i}):1\leq i\leq \ell\}$ is $k$-inconsistent. 
\end{itemize}
By compactness, this will suffice to show that $\varphi(x;y)$ has $k$-BTP.

Fix $m< \omega$ with $m>0$, and proceed by induction on $n$. The base cases $n = 0$ and $n = 1$ are trivial, since $m^{<n}_*$ is empty. 

Suppose we are given $F_0 = (a_\eta)_{\eta\in m^{<n}_*}$ satisfying the induction hypothesis. Let $b_0$ realize $p|_{MF_0}$. Since $p$ is Kim-strict, $F_0\ind^{K}_M b_0$. Let $r(z,y) = \tp(F_0b_0/M)$. 

By induction on $1\leq \ell\leq m$, we now find $(b_i, F_i)_{i<\ell}$ such that 
\begin{enumerate}
\item $F_i \equiv_M F_0$ for all $i<\ell$.
\item $b_i$ realizes $p|_{MF_j}$ if $i\leq j$.
\item $(b_i,b_{i+1},\dots,b_{\ell-1})$ starts a Morley sequence in $q$ over $MF_j$ if $i> j$.
\end{enumerate}
In the base case $\ell = 1$, $b_0$ and $F_0$ satisfy the conditions. 

Given $(b_i, F_i)_{i<\ell}$ satisfying (1)--(3) for $\ell < m$, let $b_{\ell}$ realize $q|_{M(b_i,F_i)_{i<\ell}}$. Then (3) is satisfied for $\ell+1$. Since $r(z,b_0) = \tp(F_0/Mb_0)$ does not Kim-divide over $M$ and $(b_i)_{i<\ell+1}$ starts a Morley sequence for a global $M$-invariant type, $\bigcup_{i<\ell+1}r(z,b_i)$ is consistent. Let $F_\ell$ realize this type. Then (1) is satisfied for $\ell+1$. Now since $r(F_0,z) = \tp(b_0/MF_0) = p|_{MF_0}$, $p$ is $M$-invariant, and $F_\ell \equiv_M F_0$, we have, for all $i<\ell+1$, $\tp(b_i/MF_\ell) = r(F_\ell,z) = p|_{MF_{\ell}}$, and thus (2) is satisfied for $\ell+1$. 

Having constructed $(b_i, F_i)_{i<m}$, we reindex to define the forest $(a_\eta')_{\eta\in m^{<n+1}_*}$. By (1), we can write $F_i = (a^i_\eta)_{\eta\in m^{<n}_*}$, and each $F_i$ satisfies the induction hypothesis. Set $a'_{\langle i\rangle} = b_{m-i-1}$, and $a'_{\langle i\rangle ^\frown \eta} = a^{m-i-1}_\eta$. Note that the reindexing by $(m-i-1)$ means that our sequence $(b_i, F_i)_{i<m}$ proceeds leftward in the new forest. 

A left-leaning path in the new forest begins with at most one element $b_i$ at the bottom level and is followed by some left-leaning path in $F_j$ with $i\leq j$. By (2) and induction, the reverse sequence starts a Morley sequence for $p$ over $M$. A right-veering path in the new forest may begin with elements $b_{i_1},\dots,b_{i_\ell}$ at the bottom level, with $i_1 > \dots > i_\ell$, and is followed by a right-veering path in some $F_j$ with $i_\ell > j$. By (3) and induction, the reverse sequence starts a Morley sequence for $q$ over $M$. 
\end{proof}

We now situate $\NBTP$ relative to the other tree properties.

\begin{prop}
If $T$ is $\NTP_2$, then $T$ is $\NBTP$.
\end{prop} 
\begin{proof}
Assume $\varphi(x;y)$ has $k$-$\BTP$, witnessed by $(a_\eta)_{\eta\in \omega^{<\omega}_*}$. Consider the array $(b_{i,j})_{i,j< \omega}$ with $b_{i,j} = a_{(0^i)^\frown \langle j\rangle}$, where $0^{i}$ denotes the string of length $i$ consisting of all $0$s. 

For all $f\colon \omega\to \omega$, the sequence $(\lambda_i)_{i< \omega}$ with $\lambda_i = (0^i)^\frown\langle f(i)\rangle$ is a left-leaning path. So $\{\varphi(x;b_{i,f(i)}):i<\omega\} = \{\varphi(x;a_{\lambda_i}):i<\omega\}$ is consistent. 

For all $i<\omega$, the sequence $(\rho_j)_{j< \omega}$ with $\rho_j = (0^i)^\frown \langle j\rangle$ is a right-veering path. So $\{\varphi(x;b_{i,j}):j<\omega\} = \{\varphi(x;a_{\rho_j}):j<\omega\}$ is $k$-inconsistent. 

Thus $\varphi(x;y)$ has $\TP_2$. 
\end{proof}

When $k>2$, a witness to $k$-$\BTP$ does not directly contain a witness to $\SOP_1$, but rather a variant of $\SOP_1$ with $k$-inconsistency instead of $2$-inconsistency. So for the implication from $\NSOP_1$ to $\NBTP$, we will use the following alternative characterization of $\SOP_1$ from~\cite{kaplan2017kim}.

\begin{fact}[{\cite[Proposition 2.4]{kaplan2017kim}}]\label{fact:sop1}
$T$ has $\SOP_1$ if and only if there exists $k< \omega$ and an array $(c_{i,j})_{i<\omega,j<2}$ such that:
\begin{itemize}
\item $c_{n,0}\equiv_{(c_{i,j})_{i<n,j<2}} c_{n,1}$ for all $n<\omega$.
\item $\{\varphi(x;c_{i,0}): i<\omega\}$ is consistent. 
\item $\{\varphi(x;c_{i,1}): i<\omega\}$ is $k$-inconsistent. 
\end{itemize}
\end{fact}

\begin{prop}
If $T$ is $\NSOP_1$, then $T$ is $\NBTP$.
\end{prop} 
\begin{proof}
Assume $\varphi(x;y)$ has $k$-$\BTP$, witnessed by $(a_\eta)_{\eta\in \omega^{<\omega}_*}$. Consider the binary sub-tree $(b_\eta)_{\eta\in 2^{<\omega}}$ with $b_\eta = a_{\langle0\rangle^\frown \eta}$. This tree does not witness $\SOP_1$, but it does have the following properties, which will be sufficient to obtain $\SOP_1$:
\begin{itemize}
\item For any $\rho\in 2^\omega$, $\{\varphi(x;b_{\rho|n}):n<\omega\}$ is consistent (since the corresponding sequence in our original forest is a left-leaning path). 
\item For any $\mu_1,\dots,\mu_k\in 2^{<\omega}$ such that $\mu_i^\frown\langle 1\rangle \trianglelefteq \mu_{i+1}$ for all $1\leq i<k$, $\{\varphi(x;b_{\mu_i^\frown\langle 0\rangle}):1\leq i \leq k\}$ is inconsistent (since the corresponding sequence in our original forest is a right-veering path of length $k$). 
\end{itemize}

By compactness, we can obtain a tree $(b_\eta)_{\eta\in 2^{<\kappa}}$, where $\kappa>|S_y(T)|$, which satisfies the obvious extensions of the two properties above. 

Following the proof of~\cite[Proposition 5.2]{ArtemNick}, we define $(\eta_i,\nu_i)_{i<\omega}$ in $2^{<\kappa}$ by recursion. Given $(\eta_i,\nu_i)_{i<n}$ (and setting $\eta_{-1} = \langle\rangle$ when $n = 0$), let $\mu_\alpha = {\eta_{n-1}}^\frown (1^\alpha)^\frown\langle 0\rangle$ for all $\alpha<\kappa$. Since $\kappa>|S_y(T)|$, there are $\alpha<\beta<\kappa$ such that $b_{\mu_\alpha}$ and $b_{\mu_\beta}$ have the same type over $(b_{\eta_i},b_{\nu_i})_{i<n}$. Let $\nu_n = \mu_\alpha$ and $\eta_n = \mu_\beta$. Directly from the construction, we have the following properties: 
\begin{enumerate}
\item $b_{\eta_n} \equiv_{(b_{\eta_i},b_{\nu_i})_{i<n}} b_{\nu_n}$ for all $n$. 
\item If $i < j$, then $\eta_i \vartriangleleft \eta_j, \nu_j$.
\item For all $i$, $(\eta_i\wedge \nu_i)^\frown \langle 1\rangle \trianglelefteq \eta_i$ and $(\eta_i\wedge \nu_i)^\frown \langle 0\rangle = \nu_i$. 
\end{enumerate}

Now, in the statement of Fact~\ref{fact:sop1}, set $c_{i,0} = b_{\eta_i}$ and $c_{i,1} = b_{\nu_i}$ for all $i<\omega$. We have $c_{n,0}\equiv_{(c_{i,j})_{i<n,j<2}} c_{n,1}$ by (1). Since $(\eta_i)_{i<\omega}$ is a chain in $2^{<\kappa}$ by (2), $\{\varphi(x;c_{i,0}):i<\omega\} = \{\varphi(x;b_{\eta_i}):i<\omega\}$ is consistent. And setting $\mu_i = (\eta_i\wedge \nu_i)$ for all $i$, note that by (2) and (3), $\nu_i = \mu_i^\frown \langle0\rangle$, and $\mu_i^\frown\langle 1\rangle \trianglelefteq \eta_i \trianglelefteq (\eta_j\wedge \nu_j) = \mu_j$ when $i<j$. So $\{\varphi(x;c_{i,1}):i<\omega\} = \{\varphi(x;b_{\mu_i^\frown\langle 0\rangle}):i<\omega\}$ is $k$-inconsistent. Thus $T$ has $\SOP_1$.
\end{proof}

\begin{prop}
If $T$ is $\NBTP$, then $T$ is $\NATP$. 
\end{prop}
\begin{proof}
Assume $\varphi(x;y)$ has $\ATP$, witnessed by $(a_\eta)_{\eta\in 2^{<\omega}}$. 

Define a map $e\colon \omega^{<\omega}\to 2^{<\omega}$ by recursion on the length of the input sequence: \begin{align*}
e(\langle \rangle) &= \langle \rangle\\
e(\eta^\frown \langle i\rangle) &= e(\eta)^\frown \langle 0\rangle ^\frown (1^{2i}).
\end{align*}
Note that if $\eta\trianglelefteq \nu$, then $e(\eta)\trianglelefteq e(\nu)$. 

Now define $f\colon \omega^{<\omega}\to 2^{<\omega}$ by $f(\eta) = e(\eta)^\frown\langle1\rangle$, and consider the tree $(b_\eta)_{\eta\in \omega^{<\omega}_*}$ with $b_\eta = a_{f(\eta)}$.

If $(\lambda_n)_{n< \omega}$ is a left-leaning path, we claim that $\{f(\lambda_n):n<\omega\}$ is an antichain in $2^{<\omega}$, and hence $\{\varphi(x;b_{\lambda_n}):n<\omega\} = \{\varphi(x;a_{f(\lambda_n})):n<\omega\}$ is consistent. 

So fix $n<m$ in $\omega$. Writing $\lambda_n = \eta^\frown\langle i\rangle$, we have $\eta^\frown\langle j\rangle \vartriangleleft \lambda_{n+1}$ for some $j\leq i$. Now if $\eta^\frown\langle j\rangle \vartriangleleft \nu$, then also $\eta^\frown\langle j\rangle \vartriangleleft \nu'$ whenever $\nu'$ is a descendent of $\nu$ or a descendent of a leftward sibling of $\nu$. Since $(n+1)\leq m$, it follows that $\eta^\frown\langle j\rangle \vartriangleleft \lambda_m$. Let $j'< \omega$ be such that $\eta^\frown\langle j\rangle ^\frown \langle j'\rangle \trianglelefteq \lambda_m$.

Now $f(\lambda_n) = e(\eta^\frown\langle i\rangle)^\frown \langle 1\rangle = e(\eta)^\frown \langle 0\rangle ^\frown (1^{2i+1})$. On the other hand, $f(\lambda_m) = e(\lambda_m)^\frown \langle 1\rangle$ has as an initial segment $e(\eta^\frown\langle j\rangle ^\frown \langle j'\rangle) = e(\eta)^\frown\langle 0\rangle^\frown (1^{2j})^\frown \langle 0\rangle ^\frown (1^{2j'})$. Since $2i+1\neq 2j$, $f(\lambda_n) \perp f(\lambda_m)$, as desired. 

If $(\rho_n)_{n<\omega}$ is a right-veering path, we claim that $f(\rho_n) \trianglelefteq f(\rho_{n+1})$ for all $n< \omega$. From this, it follows that the values $\{f(\rho_n):n<\omega\}$ are pairwise comparable, and hence $\{\varphi(x;b_{\rho_n}):n<\omega\}  = \{\varphi(x;a_{f(\rho_n)}):n<\omega\}$ is $2$-inconsistent. 

So fix $n<\omega$. Writing $\rho_n = \eta^\frown\langle i\rangle$, we have $\eta^\frown \langle j\rangle \trianglelefteq \rho_{n+1}$ for some $j>i$. Now $f(\rho_n) = e(\eta^\frown\langle i\rangle)^\frown\langle 1\rangle = e(\eta)^\frown \langle 0\rangle^\frown (1^{2i+1})$. On the other hand, $f(\rho_{n+1}) = e(\rho_{n+1})^\frown \langle 1\rangle$ has as an initial segment $e(\eta^\frown \langle j\rangle) = e(\eta) ^\frown \langle 0\rangle^\frown (1^{2j})$. Since $2i+1<2j$, $f(\rho_n) \trianglelefteq f(\rho_{n+1})$, as desired.

Thus $\varphi(x;y)$ has $2$-$\BTP$. 
\end{proof}

\section{Questions} \label{sec:questions}

We have left open several natural directions for future work. In our view, the main problem is to find a syntactic characterization of the theories satisfying New Kim's Lemma. We have shown that $\NBTP$ implies New Kim's Lemma, but it is open whether this implication reverses. No implication in either direction is known between New Kim's Lemma and $\NATP$. In light of Hanson's preprint~\cite{hanson2023biinvariant}, we are also interested in the relationship between New Kim's Lemma and the property $\NCTP$ explored there. 

\begin{quest}\label{q:syntax}
Is New Kim's Lemma equivalent to one or more of the syntactic properties $\NATP$, $\NBTP$, or $\NCTP$?
\end{quest}

However, it is conceivable that there simply is no syntactic property that characterizes New Kim's Lemma. One way of making this precise is to recall the following very general definition, due to Shelah.

\begin{defn}[{\cite[Definition 5.17]{shelahdonotunderstand}}]\,
\begin{itemize}
\item For $n< \omega$, an \emph{$n$-code} (for a partial type) is a pair $A = (A_+,A_-)$ of disjoint subsets of $[n] = \{0,\dots,n-1\}$. Given a formula $\varphi(x;y)$ and tuples $a_0,\dots,a_{n-1}\in \monster^y$, the partial type coded by $A = (A_+,A_-)$ is \[q_A(x) = \{\varphi(x;a_i): i\in A_+\}\cup \{\lnot \varphi(x;a_i):i\in A_-\}.\] 

\item For $n< \omega$, an \emph{$n$-pattern} (of consistency and inconsistency) is a pair $(C,I)$ of disjoint sets of $n$-codes. A \emph{finite pattern} is an $n$-pattern for some $n< \omega$. We say that a formula $\varphi(x;y)$ \emph{exhibits} the $n$-pattern $(C,I)$ if there are tuples $a_0,\dots,a_{n-1}\in \monster^y$ such that for every code $A\in C$, $q_A(x)$ is consistent, and for every code $A\in I$, $q_A(x)$ is inconsistent. 

\item A property of formulas $P$ is \emph{definable by patterns} if there is a set $\mathcal{F}$ of finite patterns such that $\varphi(x;y)$ has property $P$ if and only if $\varphi(x;y)$ exhibits every pattern in $\mathcal{F}$. 

\item A property $Q$ of theories is \emph{definable by patterns}\footnote{Shelah calls a property of theories which is definable by patterns ``weakly simply high straight''. This is a special case of a related notion that Shelah calls ``straightly defined''.} if there is a property $P$ of formulas which is defined by patterns, and $T$ has property $Q$ if and only if there is some formula $\varphi(x;y)$ which has property $P$. 
\end{itemize}
\end{defn}

Each of the properties $\TP$, $\TP_1$, $\TP_2$, $\SOP_1$, $\ATP$, and $\BTP$ considered in this paper are definable by patterns: let $\mathcal{F}$ consist of one pattern for each finite subset of the infinite pattern of consistency and inconsistency defining the property, and apply compactness. 

\begin{quest} \label{straight def}
Is the class of theories in which New Kim's Lemma fails definable by patterns?
\end{quest}

It would be nice to have a larger stock of examples of theories satisfying New Kim's Lemma. To this end, we would like it to be easier to check that New Kim's Lemma holds, and to have more constructions for producing theories satisfying New Kim's Lemma. 

\begin{quest}\label{q:onevar}
Does it suffice to show that New Kim's Lemma holds for formulas in a single free variable to establish that it holds for all formulas?  
\end{quest}

The analogous fact is known for each of the properties $\NTP$, $\NTP_1$, $\NTP_2$, $\NSOP_1$, and $\NATP$: to prove that a theory has one of these properties, it suffices to check that no formula $\varphi(x;y)$ has the corresponding property, where $x$ is a single variable. These arguments typically push against the syntactic definition of the property, so it is hard to envision what a solution to this question might look like without first resolving Question \ref{straight def}. In light of this, it makes sense to ask Question~\ref{q:onevar} with New Kim's Lemma replaced by $\NBTP$. 

The theory $\DLO_p$ examined in Section~\ref{sec:dlop} is a special case of a general construction, developed in~\cite{ArtemNick}, for ``parametrizing'' arbitrary Fra\"iss\'e limits with disjoint amalgamation. As shown in~\cite[Corollary 6.3]{ArtemNick}, the parametrization of a Fra\"iss\'e limit with a simple theory is always $\NSOP_1$. It seems likely that the arguments in Section~\ref{sec:dlop} generalize to provide a positive answer to the following question.

\begin{quest}\label{q:pfc}
Suppose $\K$ is a Fra\"iss\'e class with disjoint amalgamation, and let $\K_{\pfc}$ be the parametrized version of $\K$, as defined in~\cite[Section 6.3]{ArtemNick}. Let $T$ and $T_\pfc$ be the theories of the Fra\"iss\'e limits of $\K$ and $\K_\pfc$, respectively. If $T$ satisfies New Kim's Lemma (or if $T$ is $\NTP_2$), does $T_\pfc$ satisfy New Kim's Lemma?
\end{quest}

There is a theme in the literature that ``generic constructions'' (i.e., those involving taking a model companion) often produce properly $\NSOP_1$ theories. For example, interpolative fusion, introduced in~\cite{interpolativefusions}, is a general method for ``generically putting together'' multiple theories over a common reduct. In~\cite{interpolativefusionsii}, Tran, Walsberg, and first-named author showed that the interpolative fusion of stable theories over a stable base theory is always $\NSOP_1$ (and, under mild hypotheses, the interpolative fusion of $\NSOP_1$ theories over a stable base theory is always $\NSOP_1$). 

If theories satisfying New Kim's Lemma are to generalize $\NSOP_1$ theories in an analogous way to how $\NTP_2$ theories generalize simple theories, and how $\NIP$ theories generalize stable theories, then the following seems like a reasonable conjecture. 

\begin{quest}\label{q:if}
Does the interpolative fusion of $\NIP$ theories over a stable base theory always satisfy New Kim's Lemma?
\end{quest}

Questions \ref{q:pfc} and \ref{q:if} are also meaningful with New Kim's Lemma replaced by $\NBTP$. 

Finally, since the Kim's Lemma surveyed in Section~\ref{sec:diversity} form the cornerstones of the theories of independence in simple, $\NSOP_1$, and $\NTP_2$ theories, one might hope that a satisfying theory of Kim-independence, generalizing the theory of $\ind^f$ in $\NTP_2$ theories and of $\ind^K$ in $\NSOP_1$ theories, could be developed on the basis of New Kim's Lemma. A natural first step would be the chain condition. 

\begin{defn}
We say $\ind^K$ satisfies the \emph{chain condition over models} if whenever $M\models T$, $a\ind^K_M b$, and $I = (b_i)_{i < \omega}$ is a Morley sequence for a global $M$-invariant type extending $\tp(b/M)$, there exists $a'$ such that $a'b_i \equiv_{M} ab$ for all $i< \omega$, $I$ is $Ma'$-indiscernible, and $a'\ind^K_M I$.
\end{defn}

\begin{quest}\label{q:chaincondition}
If $T$ satisfies New Kim's Lemma, does $\ind^K$ satisfy the chain condition over models? 
\end{quest}

One motivation for this question is that $\ind^f$ satisfies the chain condition over models in $\NTP_2$ theories, see~\cite[Theorem 2.9]{benyaacovindependence} (and the chain condition is the key step in the proof of the variant of the independence theorem for $\NTP_2$ theories in that paper). The proof of the chain condition in~\cite{benyaacovindependence} uses both the Kim's Lemma for $\NTP_2$ theories and the syntactic definition of $\NTP_2$. So here again, if Question~\ref{q:chaincondition} has a positive answer, it may be necessary to first resolve Question~\ref{straight def}. 

\bibliography{ms.bib}{}
\bibliographystyle{plain}

\end{document}